\newif\iffigure
\figuretrue

\documentclass[11pt]{article}

\usepackage{fullpage}

\usepackage{amssymb}
\usepackage{amsmath, amssymb, amsthm,
	verbatim,enumerate,bbm,color,tikz,caption,subcaption,float} 
\usepackage{indentfirst}

\title{Ramsey-nice families of graphs}

\author{Ron Aharoni \thanks{Department of Mathematics, Technion, Haifa
		32000, Israel Email:
		raharoni@gmail.com}
	\and Noga Alon
	\thanks{Sackler School of
		Mathematics and Blavatnik School of Computer Science, Tel Aviv
		University, Tel Aviv 6997801, Israel. Email: nogaa@tau.ac.il.
		Research supported in part by a BSF grant, an ISF grant and a GIF
		grant.}
	\and Michal Amir
	\thanks{School of Mathematical Sciences, Raymond and Beverly
		Sackler Faculty of Exact Sciences, Tel Aviv University,
		Tel Aviv, 6997801, Israel. Email: michalamir@mail.tau.ac.il.}
	\and Penny Haxell \thanks{Department of Combinatorics and
		Optimization, University of Waterloo, Waterloo, Ontario, Canada N2L
		3G1. Email: pehaxell@uwaterloo.ca.  Partially supported by NSERC.}
	\and  Dan Hefetz
	\thanks{Department of Computer Science, Ariel University, Ariel 40700,
		Israel. Email: danhe@ariel.ac.il}
	\and  Zilin Jiang
	\thanks{Department of Mathematics, Technion -- Israel Institute of
		Technology, Haifa, 32000, Israel. Email:
		jiangzilin@technion.ac.il.}
	\and Gal Kronenberg
	\thanks{School of Mathematical Sciences, Raymond and
		Beverly Sackler Faculty of Exact Sciences, Tel Aviv University,
		Tel Aviv, 6997801, Israel. Email: galkrone@mail.tau.ac.il.}
	\and Alon Naor
	\thanks{School of Mathematical Sciences, Raymond and Beverly
		Sackler Faculty of Exact Sciences, Tel Aviv University,
		Tel Aviv, 6997801, Israel. Email: alonnaor@post.tau.ac.il.}}

\date{\today}
\allowdisplaybreaks

\theoremstyle{plain}
\newtheorem{theorem}{Theorem}[section]
\newtheorem{lemma}[theorem]{Lemma}
\newtheorem{claim}[theorem]{Claim}

\newtheorem{observation}[theorem]{Observation}
\newtheorem{corollary}[theorem]{Corollary}
\newtheorem{conjecture}[theorem]{Conjecture}

\newtheorem{question}[theorem]{Question}
\newtheorem{remark}[theorem]{Remark}
\newtheorem{definition}[theorem]{Definition}

\newcommand{\cG}{\mathcal G}

\newcommand{\cH}{\mathcal H}
\newcommand{\cF}{\mathcal F}
\newcommand{\cA}{\mathcal A}
\newcommand{\g}{g}
\newcommand{\clq}{c}

\newcommand{\md}[2]{\equiv #1~(\mathrm{mod}~#2)}

\newcommand{\mdt}[1]{\md{#1}{\mathrm{3}}}
\newcommand{\mds}[1]{\md{#1}{\mathrm{6}}}

\newcommand{\gnr}{\cG_{n,r}}
\newcommand{\hnr}{\cH_{n,r}}

\newcommand{\stm}{\setminus}
\newcommand{\sbst}{\subseteq}

\begin{document}
	
	\maketitle
	
	\begin{abstract}
		For a finite family $\cF$ of fixed graphs let $R_k(\cF)$ be the
		smallest integer $n$ for which every $k$-coloring of the edges of
		the complete graph $K_n$ yields a monochromatic copy of some
		$F\in\cF$. We say that $\cF$ is \emph{$k$-nice} if for every graph
		$G$ with $\chi(G)=R_k(\cF)$ and for every $k$-coloring of $E(G)$
		there exists a monochromatic copy of some $F\in\cF$.
		
		It is easy to see that if $\cF$ contains no forest, then it is not
		$k$-nice for any $k$. It seems plausible to conjecture that a (weak)
		converse holds, namely, for any finite family of graphs $\cF$ that
		contains at least one forest, and for all $k\geq k_0(\cF)$ (or at
		least for infinitely many values of $k$), $\cF$ is
		$k$-nice.
		
		We prove several (modest) results in support of this conjecture,
		showing, in particular, that it holds for each of the three families
		consisting of two connected graphs with $3$ edges each and observing
		that it holds for any family $\cF$ containing a forest with at most
		$2$ edges. We also
		study some related problems and disprove a conjecture by Aharoni,
		Charbit and Howard~\cite{ACH} regarding the size of matchings in
		regular $3$-partite $3$-uniform hypergraphs.
	\end{abstract}
	
	\section{Introduction}
	
	In Ramsey theory, for a $k$-tuple $(H_1,\ldots,H_k)$ of fixed
	graphs, the \emph{Ramsey number} $R(H_1,\ldots,H_k)$ is the smallest
	integer $n$ for which every coloring of $E(K_n)$ with the colors $1,
	\ldots, k$ yields a monochromatic copy
	(as a subgraph, not necessarily induced)	
	of $H_i$ in the color $i$,
	for some $1 \le i \le k$. The special case where $H_i=K_{n_i}$ for
	every $1 \le i \le k$, is the well-studied \textit{Ramsey number}
	$R(n_1,\ldots,n_k)$.
	
	Instead of considering edge-colorings of cliques, one can extend the
	question to general graphs. Bialostocki and Gy\'arf\'as~\cite{BG}
	asked for the smallest integer $n$ such that every graph $G$ with
	${\chi(G)=n}$ (rather than just $K_n$) has the aforementioned
	Ramsey-type property. More specifically, they asked for which tuples
	$(H_1,\ldots, H_k)$ of fixed graphs we have the property that for
	every graph $G$ with $\chi (G)=R(H_1,\ldots,H_k)$, and for every
	coloring of $E(G)$ with the colors $1, \ldots, k$, there exists a
	monochromatic copy
	(as a subgraph, not necessarily induced)
	of $H_i$ in the color $i$ for some $1 \leq i \leq
	k$. Such tuples are called \emph{Ramsey-nice}. When
	$H_1=\ldots=H_k=H$ we say that $H$ is \emph{$k$-Ramsey-nice}. For
	the remainder of this paper we abbreviate Ramsey-niceness simply to
	\emph{niceness}\footnote{The author of~\cite{GoodGraph} used the
		term \emph{good} instead of nice. We change it to avoid ambiguity,
		as being $k$-good usually means something else in Ramsey Theory.}.
	Note that a tuple $(H_1,\ldots,H_k)$ is not nice if there exists an
	index  $1 \leq i \leq k$ such that $H_i$ contains a cycle. Indeed,
	in his seminal paper~\cite{Erdos1959} Erd\H{o}s proved that for any
	two positive integers $\chi,g$ there exists a graph $G$ with
	$\chi(G)=\chi$ and girth greater than $g$. Therefore, if such an $i$
	exists, we can find a graph with arbitrarily large chromatic number
	that does not contain $H_i$ as a subgraph and color all of its edges
	with the color $i$.
	
	Extending a result of Cockayne and Lorimer~\cite{CockayneLorimer}
	regarding the Ramsey number of matchings, Bialostocki and
	Gy\'arf\'as~\cite{BG} proved that, for every positive integer $k$,
	the $k$-tuple $(M_1, \ldots, M_k)$ is nice, whenever $M_i$ is a
	matching
	(of any size $m_i$) for every $1 \leq i \leq k$.
	Garrison~\cite{GoodGraph} observed that every star is
	$k$-nice for every $k$. Using a Tur\'an-type argument, he proved
	that the path $P_4$ is $k$-nice for all $k$, except possibly for $k
	= 3$, and that the pair $(P_4,P_5)$ is nice. Garrison also found
	a sufficient condition for a graph $H$
	to be 2-nice, which he used to prove that $P_5$, $P_6$ and $P_7$ are
	all 2-nice.
	
	In this paper, we study niceness of families of graphs rather than
	ordered tuples. For a family of graphs $\cF$, let $R_k(\cF)$ denote
	the smallest integer $n$ for which in every $k$-coloring
	of $E(K_n)$ there exists a monochromatic copy of some $F \in \cF$.
	We say that $\cF$ is \emph{$k$-nice} if for every graph $G$ with
	$\chi(G)=R_k(\cF)$ and for every $k$-coloring of $E(G)$, there is a
	monochromatic copy of some $F \in \cF$. Note that if $H_1 = \ldots =
	H_k = H$, then $R(H_1, \ldots, H_k) = R_k(\cF)$ where $\cF = \{H\}$.
	If none of the elements in $\cF$ is a forest, then by the same
	argument as in the tuple case, $\cF$ is not $k$-nice for any $k$
	(there exist graphs with arbitrarily large chromatic number
	containing no member of $\cF$ as a subgraph). However, this argument
	does not apply if at least one $F \in \cF$ is a forest as we do not
	associate each member of $\cF$ with a specific color. More
	importantly, the number of colors $k$ used to color $E(G)$ does not
	depend on the size of the family $\cF$.
	
	In this paper we consider the following question.
	\begin{question}\label{qst:AllNice}
		Is it true that
		for any finite family of graphs $\cF$ that contains at least one
		forest, there exists a constant $k_0=k_0(\cF)$ such that $\cF$ is
		$k$-nice for all $k \geq k_0$?
	\end{question}
	
	It is easy to see that the answer to this question is ``yes" for any
	family containing a graph with at most $2$ edges. We discuss this in
	greater detail in Section~\ref{sec:concluding}. Our main focus here
	is families of connected graphs consisting of three edges each. Let
	$K_3$ be the triangle, let $P_4$ be the path on four vertices and
	let $S_3$ be the graph consisting of three edges with one common
	vertex (the star on four vertices). Note that these are the only
	connected graphs with three edges. Put $\cF_1=\{K_3\}$,
	$\cF_2=\{P_4\}$, $\cF_3=\{S_3\}$, $\cF_4=\{K_3,P_4\}$,
	$\cF_5=\{K_3,S_3\}$, $\cF_6=\{P_4,S_3\}$, $\cF_7=\{K_3,P_4,S_3\}$.
	
	Since the only member of $\cF_1$ is a cycle, $\cF_1$ is not $k$-nice
	for any $k$. As was already mentioned, it is shown
	in~\cite{GoodGraph} that $\cF_2$ is $k$-nice for every $k$, except
	possibly for $k=3$. This is proved using the known results that
	$R_k(P_4)=2k+1$ for $k\equiv 0$ or $k\mdt{2}$ (but $k\neq 3$),
	$R_k(P_4)=2k+2$ for $k\mdt{1}$, and $R_3(P_4)=6$
	(see~\cite{Bierbrauer,Irving}). The only member of $\cF_3$ is a
	star, hence the results in \cite{GoodGraph} imply that it is
	$k$-nice for all $k$. Moving on to the families that consist of two
	graphs, for the families $\cF_4$ and $\cF_5$ we can determine
	niceness for every $k$. However, the analysis for $\cF_6$ is more
	complicated, involves some divisibility conditions and requires $k$
	to be large enough in some cases. For this purpose we define the
	following sets of integers (the specific value of $\Delta_0$ will be
	determined later in Definition~\ref{def:D0}).
	
	\begin{definition}
		Let
		\begin{itemize}
			\item $\cA_0 = \{n \in \mathbb{N} \mid n\mdt{0}\}\setminus
			\{3,6,18,21,24,30,33,39,42,51,66\}$;
			\item $\cA_1 = \{n \in \mathbb{N} \mid n\mdt{1}\}$;
			\item $\cA_2 = \{n \ge \Delta_0 \mid n\mdt{2}\} \cup \{2\}$;
			\item $\cA=\cA_0\cup \cA_1\cup \cA_2$ .
		\end{itemize}
	\end{definition}
	
	Our main theorem is the following.
	
	\begin{theorem}\label{thm:main}
		Let $\cF_4, \cF_5, \cF_6$ be as defined above. Then,
		\begin{enumerate}[$(1)$]
			\item $\cF_4$ is $k$-nice for every $k$;
			\item $\cF_5$ is $k$-nice for every $k\geq 2$, but not for
			$k=1$;
			\item $\cF_6$ is $k$-nice for every $k\in \cA$.
		\end{enumerate}
	\end{theorem}

	For the remaining family $\cF_7$ we have the following result.
	\begin{theorem}\label{thm:F7}
		$\cF_7$ is $k$-nice for infinitely many integers $k$.
	\end{theorem}
	
	Even though Theorem~\ref{thm:F7} does not provide another example of
	a family for which the answer to Question~\ref{qst:AllNice} is
	affirmative, it does support the following weaker conjecture.
	
	\begin{conjecture}\label{conj:ManyNice}
		Any finite family of graphs $\cF$ that contains at least one forest
		is $k$-nice for infinitely many integers $k$.
	\end{conjecture}
	
	We can find support for this conjecture even if we do not limit
	the number of graphs in $\cF$ or their sizes, as shown
	in the following theorem.
	
	\begin{theorem}\label{thm:NiceStar}
		Let $r$ be a positive integer, and let $\cF$ be a family of graphs
		such that $K_{1,r+1} \in \cF$, and all other $F \in \cF$ contain
		at least one cycle. Then  $\cF$ is $k$-nice for infinitely many
		integers $k$.
	\end{theorem}
	
	The proofs of Theorems~\ref{thm:main},~\ref {thm:F7}
	and~\ref{thm:NiceStar} appear in Section~\ref{sec:nice}.\\

	The main ingredient in the proof of the last item of
	Theorem~\ref{thm:main} is the answer to the  following question: what
	is
	the maximum possible chromatic number of a graph obtained by taking
	the union of $r$ \emph{triangle factors} on the same set of
	vertices (where the different factors are not necessarily edge disjoint)?   A triangle factor is a graph in which every connected
	component is a triangle. We prove the following.
	
	\begin{theorem}\label{thm:union}
		The maximum possible chromatic number of a graph
		obtained by taking the union of $r$ triangle factors is:
		\begin{enumerate}[$(i)$]
			\item $2r+1$ for $r\in \cA_1$;
			\item $2r$ for $r\in \cA_0$;
			\item $2r-1$ for $r\in \cA_2$.
		\end{enumerate}
	\end{theorem}
	
	This theorem settles the aforementioned question for all but a finite
	number of values of $r$. However, the set of integers for which we do
	not know the answer to that question, includes values as small as $3$.
	Indeed, for $r=3$, the following question was suggested by
	Gy\'arf\'as~\cite{Gy}.
	
	\begin{question}\label{qst:union3}
		Suppose that $G$ is the union of three triangle factors on the same
		set of vertices. Is $G$ 5-colorable?
	\end{question}
	
	We discuss this topic (including the proof of
	Theorem~\ref{thm:union}) in Section~\ref{sec:trianglefactors}, and
	in addition show that the answer to Question~\ref{qst:union3} is
	``yes" if a conjecture by Molloy and Reed in~\cite{MolloyReed} is
	affirmed. In fact, the affirmation of either their conjecture or a
	conjecture by Borodin and Kostochka in~\cite{BorodinKostochka} will
	decrease the number of values
	of $r$ not covered by Theorem~\ref{thm:union} to at most eleven.\\
	
	The problems discussed in Section~\ref{sec:trianglefactors}, and
	Question~\ref{qst:union3} in particular, are also related to
	problems on matchings in hypergraphs in the following sense. Let
	$\gnr$ be the family of all graphs obtained by taking the union of
	$r$ triangle factors, each containing $n$ triangles, on the same set
	of $3n$ vertices. Let $\hnr$ be the family of all $r$-equipartite
	$r$-uniform 3-regular hypergraphs on $rn$ vertices. That is, every
	$\cH \in \hnr$ has a vertex set $V(\cH) = V_1 \cup \ldots \cup V_r$,
	each part satisfies $|V_i| = n$, every hyperedge has the form
	$\{v_1,\ldots,v_r\}$ where $v_i \in V_i$ for every $i$, and every
	vertex is contained in exactly 3 hyperedges (repeated hyperedges contribute with their
	multiplicities). The (simple) proof of the following lemma appears in
	the next section.
	
	\begin{lemma} \label{lem:equiv}
		$$\max_{G \in \gnr} \chi(G) = \max_{\cH \in \hnr}\chi'(\cH).$$
	\end{lemma}
	
	Note that every color class in any proper edge coloring of a
	hypergraph is a matching. Therefore, by Lemma~\ref{lem:equiv}, a low
	chromatic number of members of $\gnr$ implies a low chromatic index of
	the corresponding members of $\hnr$, which, in turn, by the
	pigeonhole principle, implies the existence of a large matching in
	these hypergraphs. In particular, if the answer to
	Question~\ref{qst:union3} is ``yes", and so $\chi(G) \le 5$ for
	every $n$ and for every $G \in \mathcal G_{n,3}$, then the following
	holds:
	In every $3$-regular $n\times n\times n$, 3-partite 3-uniform
	hypergraph (we define an $n\times n\times n$ hypergraph to be a 3-equipartite 3-uniform hypergraph on $3n$ vertices), there exists a matching of size at least $\left\lceil
	\frac 35 n\right\rceil$.
	This is proved in~\cite{CKW} by Cavenagh, Kuhl and Wanless for such
	hypergraphs
	assuming they have no repeated edges.
	
	One can consider a similar question in a more general setting: what
	is the largest matching guaranteed in any $d$-regular $r$-partite
	$r$-uniform hypergraph? Aharoni, Charbit and Howard conjectured the
	following
	for $r=3$.
	
	\begin{conjecture}[Conjecture 9.3 in~\cite{ACH}]\label{conj:Aharoni}
		In any $d$-regular $n\times n\times n$, $3$-partite, $3$-uniform
		hypergraph
		not containing repeated edges, there exists a matching of size at
		least $\left\lceil \frac {d-1}d n\right\rceil $.
	\end{conjecture}
	
	We disprove this conjecture by a large margin in
	Section~\ref{sec:matchings} by showing that, even for arbitrarily
	large $d$, there are such hypergraphs containing no matching of size
	larger than $2n/3$. A special case of a conjecture in~\cite{AK}, if
	true,
	implies that $2n/3$ is also a lower bound, even if the hypergraph is
	not
	$3$-partite. It is also  known (see, e.g.,~\cite{AKS}), that if the
	hypergraph is linear, that is, contains no two edges that share more
	than one common
	vertex, then there is always a matching of size at least
	$\left(1 - O\left(\frac{\log^{3/2}d}{\sqrt d }\right)\right)n$.
	For large values of $r$, we show that there exist $r$-uniform $r$-partite $d$-regular hypergraphs on $n$ vertices, for arbitrarily large $d$, such that the largest matching covers only $(1+o(1))\frac{n}{r}$ vertices, which is asymptotically tight.
	
	\section{Preliminaries and notation}
	
	For every positive integer $k$ we use $[k]$ to denote the set
	$\{1,2,\ldots,k\}$.
	Our graph-theoretic notation is standard and follows that
	of~\cite{West}. In particular, we use the following.
	
	For a graph $G=(V,E)$ let $\overline{G}=(V,\overline{E})$ denote the
	complement graph of $G$, that is, $\overline{E}=\{uv~|~u\neq
	v\in V,~ uv \notin E\}$. For a set of vertices $U \subseteq V(G)$,
    the subgraph of $G$ induced by $U$ is denoted by $G[U]$.
	For a subset $U \subseteq V$, let $N_G(U) = \{v \in V \setminus U
	\mid \exists u\in U \textrm{ such that } uv \in E(G)\}$ denote the
	external neighborhood of $U$ in $G$. For a vertex $v \in V$ we
	abbreviate $N_G(\{v\})$ to $N_G(v)$ and let $d_G(v) = |N_G(v)|$
	denote the degree of $v$ in $G$.
	The maximum degree and the minimum degree in $G$ are denoted by $\Delta(G)$ and $\delta(G)$,
	respectively. Often, when there is no risk of ambiguity, we omit the
	subscript $G$ in the above notation.
	
	The size of a largest clique in $G$ is denoted by $\omega(G)$. For
	an integer $k$, the \emph{$k$-core} of a graph $G$ is the (unique)
	maximal subgraph of $G$ in which all vertices have degree at least
	$k$. If no such subgraph exists we say that $G$ has an empty
	$k$-core.
	
	A \emph{$k$-coloring} of a graph $G$ is a function $f : V(G) \to
	[k]$. A coloring $f$ of $G$ is called \emph{proper} if $f(v) \neq
	f(u)$ for every pair of adjacent vertices $u,v \in V(G)$. The graph
	$G$ is called \emph{$k$-colorable} if there exists a proper
	$k$-coloring of $G$. The \textit{chromatic number} of a graph $G$,
	denoted by $\chi(G)$, is the minimal $k$ for which $G$ is
	$k$-colorable. Similarly, a \emph{$k$-edge-coloring} of $G$ is a
	function $f:E(G)\to [k]$, it is proper if $f(e_1) \neq f(e_2)$ for
	every pair of intersecting edges $e_1, e_2$, and the
	\textit{chromatic index} of a graph $G$, denoted by $\chi'(G)$, is
	the minimal $k$ for which there exists a proper $k$-edge-coloring of
	$G$. A graph $G$ is called \emph{$s$-critical} if $\chi(G)=s$ and
	$\chi(H)<s$ for every  proper subgraph $H$ of $G$.
	
	A set of graphs $\{G_1,\ldots,G_k\}$ is a \emph{covering} of a graph
	$G$ if $E(G) \sbst \bigcup_{i=1}^{k} E(G_i)$. Such a set is called a
	\emph{decomposition} of $G$ if, in addition, $E(G_i) \sbst E(G)$ for
	every $1 \leq i \leq k$, and all the graphs in this set are pairwise
	edge-disjoint. For a fixed graph $H$ on $h$ vertices, and for an
	$n$-vertex graph $G$ with $n$ divisible by $h$, an \emph{$H$-factor}
	of $G$ is a collection of $n/h$ copies of $H$ whose vertex sets
	partition $V(G)$, and each copy of $H$ is a subgraph of $G$.  We say that a graph $G$ is a union of $\ell$
	$H$-factors if there exist $H_1,\dots,H_\ell$ such that
	$V(H_i)=V(G)$ for every $i\in[\ell]$,
	$E(G)=\bigcup_{i=1}^{\ell}E(H_i)$, and every $H_i$ is an $H$-factor
	of $G$.
	
	We now present several theorems and observations which will be
	useful in our proofs, starting with the proof of
	Lemma~\ref{lem:equiv} mentioned in the previous section.
	
	\begin{proof} [Proof of Lemma~\ref{lem:equiv}]
		We in fact prove a stronger result. Namely, we show that there exists
		a
		bijection between pairs $(G, f_G)$ and $(\cH, f_\cH)$ where
		$G \in \gnr$, $f_G$ is a coloring of $G$, $\cH \in \hnr$ and $f_\cH$
		is an edge coloring of $\cH$, such that $f_G$ and $f_\cH$ use the
		same number of colors and $f_G$ is proper if and only if $f_\cH$ is
		proper.
		
		We first describe a bijection between $\gnr$ and $\hnr$. Consider a
		graph $G\in \gnr$ obtained by a union of $r$ triangle factors
		$H_1,\ldots,H_r$, each factor containing $n$ triangles. We construct
		a hypergraph $\cH$ with parts $V_1,\ldots, V_r$, each of size $n$, in
		the following way. For every $1 \leq i \leq r$ and for every triangle
		$T\in H_i$
		we have a vertex $v_T \in V_i$. For every vertex $v \in V(G)$ we
		have an edge $e_v\in E(\cH)$ consisting of the $r$ vertices
		representing
		the
		$r$ triangles containing $v$. It is easy to see that the
		constructed hypergraph $\cH$ is indeed a member of $\hnr$ and that
		this is a bijection (in fact, $G$ is the line graph of $\cH$).
		
		Given $G$ and the corresponding $\cH$, the bijection between
		colorings $f_G$ of $G$ and edge colorings $f_\cH$ of $\cH$ is the
		obvious one: $f_G(v) = f_\cH(e_v)  $. Finally,
		note that $f_G$ is a proper coloring if and only if $f_\cH$ is a
		proper edge coloring. Indeed, any two vertices $u,v$ in $G$ are
		adjacent if and only if they belong to the same triangle $T$ in (at
		least) one of the factors, which happens if and only if
		$v_T \in e_u \cap e_v$.
	\end{proof}
	
	\begin{observation}\label{obs:clique}
		Let $G$ be a graph on $n+1$ vertices with chromatic number $n$. Then
		$G$ contains $K_n$ as a subgraph.
	\end{observation}
	
	\begin{proof}
		Let $V(G)=\{v_1,\ldots,v_{n+1}\}$ and assume for a contradiction
		that $G$ does not contain a copy of $K_n$ as a subgraph. We claim
		that $\overline{G}$ must contain either a triangle or a matching of
		size two. Indeed, otherwise all edges of $\overline{G}$ must share
		some vertex $v_i \in V(G)$. But then $G - v_i$ is a clique on $n$
		vertices, a contradiction. Now, if $\overline{G}$ contains a
		triangle we can color the three vertices of the triangle with one
		color, and color each of the remaining $n-2$ vertices with a unique
		new color.  In the second case, there exist two independent edges
		$e_1, e_2 \in E(\overline{G})$. We can color the two endpoints of
		$e_1$ with one color, the two endpoints of $e_2$ with a second
		color, and finally assign a unique new color to each of the
		remaining $n-3$ vertices. In either case we got a proper
		coloring of $G$ with only $n-1$ colors, a contradiction.%
	\end{proof}
	
	\begin{observation}\label{obs:coloring}
		Let $G$ be a graph and let $d$ be  an integer such that the $d$-core
		of
		$G$
		is $d$-colorable. Then $G$ is $d$-colorable.
	\end{observation}
	
	\begin{proof}
		Consider the following vertex deletion algorithm to obtain the
		$d$-core
		of
		$G$: starting with $G_0 := G$, for every $i \ge 0$, if $G_i$ contains
		a
		vertex of degree less than $d$, we choose one such vertex
		arbitrarily, denote it by $v_i$ and let $G_{i+1} := G_i
		- v_i$. The algorithm terminates with $G_k \subseteq G$ (for some
		$k$), when there are no more vertices of degree less than $d$. It is
		easy to see and well known that $G_k$ is the (possibly empty) $d$-core of
		$G$,
		regardless of the arbitrary choices made during the process.
		
		Note that $d_{G_i}(v_i) < d$ for every $0 \le i \le k-1$ and thus
		any proper $d$-coloring of $G_{i+1}$ can be trivially extended to a
		proper $d$-coloring of $G_i$. Since there exists such coloring for
		$G_k$ by assumption, by greedily coloring $v_{k-1},\ldots, v_0$ in
		this order we obtain a proper $d$-coloring of $G$.
	\end{proof}

	The following are two fundamental theorems in graph theory, by
	K\"onig and by Brooks.
	
	\begin{theorem}[\textbf{\cite{konig}, K\"onig's theorem}]
		\label{thm:konig}
		Every bipartite multigraph $G$ has a proper edge-coloring with
		$\Delta(G)$ colors.
	\end{theorem}
	
	\begin{theorem}[\textbf{\cite{brooks}, Brooks' theorem}]
		\label{thm:brooks}
		Let $G$ be a connected simple graph with $\Delta(G)= \Delta$. Then
		$\chi(G)\leq\Delta$ unless $G$ is a complete graph or an odd cycle,
		in which case $\chi(G)=\Delta + 1$.
	\end{theorem}
	
	In Section~\ref{sec:trianglefactors} we consider the problem of
	covering
	complete graphs by triangle factors. The following theorem deals
	with such graphs, where the number of vertices is divisible by six
	(see~\cite{ResolvableCoverings} and page 386 of~\cite{stinson}).
	\begin{theorem}[\cite{ResolvableCoverings}]\label{thm:0mod6}
		For $n\geq 18$, if $n\mds{0}$ and $n \notin
		\{36,42,48,60,66,78,84,102,132\}$, then one can cover the edges of
		$K_n$
		with $n/2$ triangle factors. On the other hand, the edges of $K_{12}$ cannot be
		covered by six triangle factors.
	\end{theorem}
	
	On the same topic, the following result is an immediate corollary of
	the
	work of Kirkman from 1847~\cite{kirkman}, which was subsequently
	completed
	by Ray-Chaudhuri and Wilson~\cite{RCW1}.
	
	\begin{theorem}\label{thm:3mod6}
		$K_n$ can be decomposed into triangle factors if and only if
		$n\mds{3}$.
	\end{theorem}
	
	For the proofs of Theorems~\ref{thm:F7} and~\ref{thm:NiceStar}, we
	need a more general result on decompositions of complete graphs.
	
	\begin{theorem}[Theorem 1.3 in \cite{DL}]\label{thm:Hdecomp}
		Let $H$ be a simple graph on $h$ vertices with degree
		sequence $d_1,\dots,d_h$ and average degree $\bar{d}$. Then there
		exists a decomposition of $K_n$ into $H$-factors for every
		sufficiently large $n$ satisfying the following.
		\begin{enumerate}[$(a)$]
			\item $n \md{0}{h}$;
			\item $n-1 \md {0}{\gamma}$, where $\gamma$ is the smallest
			positive
			integer such that $$\left(\gamma, \gamma / \bar{d}\right) \in
			span^{}_\mathbb{Z}\left\{\left(d_i, 1\right) \mid i \in [h]\right\}.$$
		\end{enumerate}
	\end{theorem}

	In Section~\ref{sec:trianglefactors} we study graphs whose chromatic
	number and maximum degree are very close. As part of our
	proof we rely on the work of Molloy and Reed in~\cite{MolloyReed}.
	In particular, we use the following.
	
	\begin{definition}
		$k_\Delta$ is the maximum integer $k$ such that $(k + 1)(k + 2) \leq
		\Delta$.
	\end{definition}
	
	\begin{definition}
		A $c$-reducer $R=(C,S)$ of a graph $G$ consists of a clique $C$ on
		$c - 1$ vertices and a stable set $S$ such that every vertex of $C$
		is adjacent to all the vertices of $S$ but none of $V(G) \setminus
		(S\cup C)$.
	\end{definition}

	\begin{theorem}[Theorem 5 in~\cite{MolloyReed}]\label{thm:MolloyReed}
		There is an absolute constant $\Delta_1$ such that for any
		$\Delta\geq \Delta_1$ and $c\geq \Delta-k_\Delta$,
		if $G$ is a graph with maximum degree at most $\Delta$,  $\chi(G) =
		c+1$, and either
		\begin{enumerate}[$(i)$]
			\item  $c\geq \Delta-k_\Delta+1$, or
			\item G has no $c$-reducer,
		\end{enumerate}
		then there is some vertex $v\in V(G)$ such that the subgraph induced
		by $\{v\}\cup N(v)$ has chromatic number $c+1$.
	\end{theorem}
	
	In the same paper, Molloy and Reed conjectured that in fact there is
	no need for the condition $\Delta\geq \Delta_1$ in
	Theorem~\ref{thm:MolloyReed}.
	
	\begin{conjecture}[Conjecture 6 in~\cite{MolloyReed}]
		\label{conj:MolloyReed}
		Theorem~\ref{thm:MolloyReed} holds for every $\Delta$, that is,
		one can take $\Delta_1 = 1$.
	\end{conjecture}
	
	Another paper that deals with similar topics is that of
	Borodin and Kostochka~\cite{BorodinKostochka}, in which they
	conjectured the following.

	\begin{conjecture}[\cite{BorodinKostochka}]\label{conj:BorodinKostochka}
		Let $G$ be a graph with $\Delta(G)\geq 9$ and $\omega(G)<\Delta(G)$.
		Then $\chi(G)<\Delta(G)$.
	\end{conjecture}
	
	Although this conjecture remains unproven, Reed~\cite{reed1999}
	proved it for sufficiently large $n$.
	
	\begin{theorem}[Theorem 4 in~\cite{reed1999}]\label{thm:Reed}
		There is a constant  $\Delta_2$ such that if $G$ is a graph with
		$\Delta(G)\geq \Delta_2$ and $\omega(G)<\Delta(G)$, then
		$\chi(G)<\Delta(G)$. Furthermore, $\Delta_2 \le 10^{14}$.
	\end{theorem}
	
	In our proofs we can use either Theorem~\ref{thm:MolloyReed} or
	Theorem~\ref{thm:Reed}. Since they both involve a large lower bound
	on the maximum degree in graphs, we use implicitly the one with the
	lower such bound. To this end, we now define the integer $\Delta_0$
	that appeared in the definition of $\cA_2$.
	
	\begin{definition}\label{def:D0}
		Let $\Delta_0 = \min\{\lceil\Delta_1/2\rceil,
		\lceil\Delta_2/2\rceil\}$, where $\Delta_1$ and $\Delta_2$ are as
		defined in Theorems~\ref{thm:MolloyReed} and~\ref{thm:Reed},
		respectively.
	\end{definition}

	\begin{remark}
		\emph{Note that if either Conjecture~\ref{conj:MolloyReed} or
			Conjecture~\ref{conj:BorodinKostochka} is affirmed, then it follows
			that $\cF_6$ is $k$-nice for every $k \mdt{2}$. In this case, only
			11 values of $k$ (all of them divisible by three) will remain
			not covered by Theorem~\ref{thm:main}. Indeed, the affirmation of
			Conjecture~\ref{conj:MolloyReed} will simply mean that $\Delta_0 =
			1$, and for Conjecture~\ref{conj:BorodinKostochka} we will have that
			$\Delta_0 = 5$. In either case, by definition we will get $\cA_2 =
			\{n \in \mathbb{N} \mid n \mdt{2}\}$. }
	\end{remark}

	\section{The chromatic number of the union of triangle factors}
	\label{sec:trianglefactors}
	
	In this section we discuss the chromatic number of graphs obtained
	by a union of triangle factors. It will in fact be more convenient
	to discuss \emph{generalized} triangle factors.
	
	\begin{definition}
		A graph $G$ is called a \textit{generalized triangle factor} if
		every connected component of $G$ is a subgraph of a triangle, i.e.,
		a triangle, a path of length 2, an edge, or a vertex.
	\end{definition}
	
	Let
	$$\chi_r^{} := \max\{\chi(G) \mid \text{$G$ is a union of $r$
		generalized triangle factors on the same set of vertices} \}$$
	and
	$$\chi_r^{*} := \max\{\chi(G) \mid \text{$G$ is a union of $r$ triangle
		factors on the same set of vertices}\}.$$
	
	\begin{claim}\label{cl:generalized}
		$\chi_r^{} = \chi_r^{*}$ for every integer $r$.
	\end{claim}
	
	\begin{proof}
		Clearly, $\chi_r^*\leq \chi_r^{}$ as every triangle factor is also a
		generalized triangle factor. The other direction follows from the
		simple
		fact that every union of $r$ generalized triangle factors  is a
		subgraph
		of
		a union of $r$ triangle factors (not necessarily on the same vertex
		set).
	\end{proof}
	
	%

	By Claim~\ref{cl:generalized} we can shift our focus to generalized
	triangle factors, which from now on will be referred to simply as
	factors, whereas triangle factors will be referred to  as \emph{proper}
	factors.
	We may also assume that the graphs we discuss are connected, since we
	can always
	restrict our analysis to a connected component that has the same
	chromatic number as the whole graph. We separate our analysis of
	$\chi_r^{}$ into three cases, according to the residue  of $r$ mod  3.
	We thus use the following notation for $i \in
	\{0,1,2\}$, in order to make our arguments easier to follow, we write $\chi_{r, i}^{} $ for $\chi_{r}^{} $ whenever $r\equiv i\ \text{(mod 3)}$. Thus, any claim about $\chi_{r, i}^{}$ should be interpreted as a claim about $\chi_{r}^{}$ only for those $r$ such that  $r\equiv i\ \text{(mod 3)}$.
	
	We now prove a few useful claims to be used later in the proof of
	Theorem~\ref{thm:union}.
	
	\begin{claim}\label{cl:6gen}
		The edges of $K_6$ cannot be covered by three factors.
	\end{claim}
	
	\begin{proof}
		Consider a union of three factors on six vertices. In order to cover
		all fifteen edges of $K_6$, at least one factor has to be of size at
		least five. Therefore, one of the components in this factor has to
		be a triangle, and clearly we may as well assume that the other one
		is also a triangle. It is easy to see that any other factor contains
		at most four edges not contained in the first one. All in all, we
		can cover at most fourteen edges.
	\end{proof}
	
	\begin{claim}\label{cl:7minus3}
		It is impossible to cover 18 of the edges of $K_7$ with three
		factors.
	\end{claim}
	
	\begin{proof}
		Every factor on seven vertices contains at most six edges.
		Therefore, in order to cover 18 of the edges of $K_7$, all factors
		should be pairwise edge-disjoint and each factor should consist of
		two triangles and one isolated vertex. However, every two such
		factors have at least one common edge.
	\end{proof}

	\begin{claim}\label{cl:K2rNotCover}
		Let $K_k^-$ be the complete graph on $k$ vertices, missing one edge.
		The edges of $K_{2r}^-$
		cannot be covered by $r$ factors for any $r
		\mdt{2}$ except $r=2$.
	\end{claim}
	
	\begin{proof}
		Let $r \mdt{2}$, $r \ge 5$, be an integer and let $G$ be a graph on
		$2r$ vertices which is the union of $r$ factors $H_1,\ldots,H_r$. We
		need to show that $G$ does not contain a copy of  $K_{2r}^-$ as a
		subgraph. We show that $G \neq K_{2r}^-$, and a very similar
		argument, whose details we omit, shows that $G \neq K_{2r}$ as well.
		First note that $|E(K_{2r}^-)|=r(2r-1)-1$, and that there are
		exactly two vertices in $K_{2r}^-$ with degree $2r-2$, while the
		rest of the vertices have degree $2r-1$.
		
		For every $1 \leq i \leq r$, let $t_i$ denote the number of
		non-triangle
		connected
		components in $H_i$. Since $2r \mdt{1}$, it follows that $t_i > 0$ for
		every $i$.
		Note that the number of edges in every such component is one less than
		the number of its vertices. Hence, $|E(H_i)| = 2r - t_i$. If there
		exist
		$i \neq
		j$ for which $t_i, t_j\geq 2$, then $|E(G)|\leq 2(2r-2)+(r-2)(2r-1)
		= r(2r-1)-2$, so $G \neq K_{2r}^-$ in this case. Assume then that
		$t_i = 1$ for all $i$ but at most one. Since $2r \mdt{1}$, the only
		non-triangle
		connected component in each such factor must be an isolated vertex.
		If some vertex $v$ is isolated in two different factors, then
		$d_G(v) = \sum_{i=1}^r d_{H_i}(v) \leq 0 \cdot 2 + 2 \cdot (r-2) =
		2r-4$,
		and thus $G \neq K_{2r}^-$. Otherwise, there are $r-1 > 2$ different
		vertices
		with degree at most $2r-2$, and, once again, $G \neq K_{2r}^-$.
	\end{proof}

	\begin{claim}\label{cl:ProperFactors}
		Let $G = (V,E)$ be a connected graph obtained by a union of proper
		factors on the same set of vertices, let $A \subseteq V$ satisfy $|A| \mdt{k}$ for
		$k\in\{1,2\}$ and let $B = A \cup N_G(A)$. Then
		\begin{enumerate}[$(a)$]
			\item $|N_G(A)| \ge 3-k$;
			\item If $|N_G(A)| = 3-k$, then $B = V$.
		\end{enumerate}
	\end{claim}
	
	\begin{proof}
		Let $H$ be a proper factor of $G$, and let $H'= H[B]$. Note that every triangle in $H$ containing at least one vertex of
		$A$ is contained entirely in $H'$, thus implying $(a)$. If
		$|N_G(A)| = 3-k$ then $H'$ is necessarily a proper factor. Since
		this is true for any proper factor $H$, it follows that $E(B, V\stm
		B) = \emptyset$. Since, moreover, $G$ is connected, this implies
		$(b)$.
	\end{proof}
	
	\begin{claim}\label{cl:noReducer}
		Let $G$ be a connected graph obtained by a union of $r \in \{3,5\}$
		proper factors such that ${\Delta(G) = 2r}$. Then $G$ has
		no $(2r-1)$-reducer.
	\end{claim}
	
	\begin{proof}
		Assume for a contradiction that $G$ contains a $(2r-1)$-reducer $R =
		(C,S)$ with $|C| = 2r-2$, and note that $S = N_G(C)$.
		
		For $r=3$ we have $|C| = 4 \mdt{1}$. It follows by
		Claim~\ref{cl:ProperFactors}(a) (applied with $A = C$) that $|S| \ge
		2$.
		If $|S| = 2$, then $|V(G)| = 6$ by Claim~\ref{cl:ProperFactors}(b),
		contrary to our assumption that $\Delta(G) = 6$. On the other hand, if
		$|S| > 2$, then $C$ and $3$ vertices of $S$ form a copy of $K_7$ minus three edges,
		contrary to the assertion of Claim~\ref{cl:7minus3}.
		
		Similarly, for $r=5$ we have $|C| = 8 \mdt{2}$, and thus
		Claim~\ref{cl:ProperFactors}(a) implies that $|S| \ge 1$. If $|S| =
		1$, then $|V(G)| = 9$ by Claim~\ref{cl:ProperFactors}(b), contrary
		to our assumption that $\Delta(G) = 10$. On the other hand, if $|S|
		> 1$, then $G$ contains a copy of $K_{10}$ minus an edge, contrary
		to the assertion of Claim~\ref{cl:K2rNotCover}.
	\end{proof}
	
	Before we state and prove our next claim, recall that $\Delta_1$
	is the constant appearing in Theorem~\ref{thm:MolloyReed}.
	
	\begin{claim} \label{cl:356}
		For $r \in \{3,5,6\}$, if $\Delta_1 \le 2r$, then $\chi_r < 2r$.
	\end{claim}
	
	\begin{proof}
		Let $r \in \{3,5,6\}$, let $G$ be a union of $r$ factors, and assume
		for a contradiction that $\Delta_1 \le 2r$ but $\chi(G) \ge 2r$.
		By
		removing edges as necessary we may assume $\chi(G)=2r$.
		We may also assume that all $r$ factors are proper by
		Claim~\ref{cl:generalized}.
		In order to
		prove the claim, we show
		that $G$ must contain a copy of $K_{2r}$, contrary to the assertion
		of Claim~\ref{cl:6gen} (for $r = 3$), Claim~\ref{cl:K2rNotCover}
		(for $r = 5$), and Theorem~\ref{thm:0mod6} (for $r = 6$). Indeed, if
		$\Delta(G) < 2r$, then $G = K_{2r}$ by Brooks' theorem
		(Theorem~\ref{thm:brooks}). Assume then that $\Delta(G) = 2r$.  For
		$r = 6$ we then have $k_\Delta = 2$, and thus condition~$(i)$ of
		Theorem~\ref{thm:MolloyReed} (here $c=2r-1$) holds in this
		case. For $r = 3$ and $r = 5$, it follows by
		Claim~\ref{cl:noReducer} that condition~$(ii)$ of
		Theorem~\ref{thm:MolloyReed} (here, again, $c=2r-1$) holds.
		In either case, we can apply Theorem~\ref{thm:MolloyReed} to deduce
		that there exists a vertex $v \in V(G)$ for which $\chi(G') = 2r$,
		where $G' = G[\{v\} \cup N_G(v)]$. Since $d_G(v) \le \Delta(G) =
		2r$, it follows that $|V(G')| \le 2r + 1$ and thus, by
		Observation~\ref{obs:clique}, there exists a clique of size $2r$ in
		$G' \subseteq G$.
	\end{proof}
	
	We are now ready to prove Theorem~\ref{thm:union}. As previously noted,
	we
	partition the proof into three cases, according to the residue  of
	the number of factors mod 3. We prove each case in a separate claim.
	
	\begin{claim}\label{cl:ChiUpperBound}
		$\chi_{r,1}^{}=2r+1$, and $\chi_r^{} \le 2r$ for every $r \not\in
		\cA_1$.
	\end{claim}
	
	\begin{proof}
		If $G$ is a union of $r$ factors $H_1,\ldots,H_r$, then
		$\Delta(G)\leq 2r$, and therefore $\chi(G)\leq 2r+1$. For $r=1$ the
		upper bound is trivially achieved as $\chi(K_3) = 3$. For every $r >
		1$, Brooks' theorem (Theorem~\ref{thm:brooks}) implies that $\chi(G)
		= 2r+1$ if and only if $G = K_{2r+1}$. Since there are $r(2r+1)$
		edges in $K_{2r+1}$, and since $|E(H_i)| \le |V(G)|$ for every $i$,
		the only way to obtain $G = K_{2r+1}$ is if each $H_i$ is a proper
		factor, and each edge in $G$ is covered by the $H_i$'s exactly once.
		In other words, $H_1,\ldots,H_r$ form a decomposition of $G$ into
		proper factors. By Theorem~\ref{thm:3mod6}, this is possible if and
		only if $|V(G)|\mds{3}$, or, in terms of $r$, if and only if
		$r\mdt{1}$.
	\end{proof}
	
	\begin{corollary}\label{cor:chi0mod3}
		$\chi_{r,0}^{}=2r$ for every $r \in \cA_0$.
	\end{corollary}
	
	\begin{proof}
		Claim~\ref{cl:ChiUpperBound} implies $\chi_{r,0}^{} \le 2r$. On the other
		hand, for every $r\in \cA_0$ the edges of $K_{2r}$ can be covered
		with $r$ factors by Theorem~\ref{thm:0mod6}, thus the upper bound is
		tight in this case.
	\end{proof}

	\begin{claim}\label{cl:chi2mod3}
		$\chi_{r,2}^{}\in  \{2r-1, 2r\}$ for every $r \mdt{2}$, and
		$\chi_{r,2}^{}=2r-1$ for every $r \in \cA_2$.
	\end{claim}
	
	\begin{proof}
		The upper bound $\chi_{r,2}^{} \leq 2r$ is an immediate corollary of
		Claim~\ref{cl:ChiUpperBound}. Observe that $(r-1) \mdt{1}$. Therefore,
		by Claim~\ref{cl:ChiUpperBound}, we can use $r-1$ of the factors to
		build a graph $G$ with $\chi(G) = 2(r-1)+1 = 2r - 1$
		(in fact, $G = K_{2r-1}$ in this case). This proves the lower bound
		$\chi_{r,2}^{} \geq 2r-1$.
		
		We now prove the second statement of the claim and begin with the
		special case $r = 2$. For convenience we consider only proper
		factors on $3n$ vertices for arbitrary $n$ (this is allowed by
		Claim~\ref{cl:generalized}), and so by Lemma~\ref{lem:equiv} we have
		$\chi_2^{}=\max_{\cH \in \cH_{n,2}}\chi'(\cH)$. Note that for every
		$n$, every $\cH\in \cH_{n,2}$ is in fact a bipartite (multi)graph
		with $\Delta(\cH)=3$. By K\"onig's theorem
		(Theorem~\ref{thm:konig}), every such hypergraph satisfies
		$\chi'(\cH)=3$.
		
		Now let $r \in \cA_2\setminus\{2\}$, let $G$ be a union of $r$
		factors, and assume for a contradiction that $\chi(G) = 2r$. In
		order to complete the proof of Claim~\ref{cl:chi2mod3}, we show
		that, contrary to the assertion of Claim~\ref{cl:K2rNotCover}, $G$
		must contain a copy of $K_{2r}$. Indeed, if $\Delta(G) < 2r$, then
		$G = K_{2r}$ by Brooks' theorem (Theorem~\ref{thm:brooks}). Assume
		then that $\Delta(G) = 2r$. If $2r \ge \Delta_2$ then
		Theorem~\ref{thm:Reed} implies that $\omega(G) \ge \Delta(G) = 2r$.
		Otherwise, $2r \ge \Delta_1$ by the definitions of $\Delta_0$ and
		$\cA_2$. Claim~\ref{cl:356} then implies that $r > 5$ and so for
		$\Delta = 2r$ we have $k_\Delta \ge 2$, and thus condition~$(i)$ of
		Theorem~\ref{thm:MolloyReed} (here $c = 2r - 1$) holds.
		Hence, there exists a vertex $v \in V(G)$ for which $\chi(G') = 2r$,
		where $G' = G[\{v\} \cup N_G(v)]$. Since $d_G(v) \le 2r$, it follows
		that $|V(G')| \le 2r + 1$ and thus, by Observation~\ref{obs:clique},
		there exists a clique of size $2r$ in $G' \subseteq G$.
	\end{proof}
	
	\begin{remark}\label{rem:LowerBoundClique}
		Observe that in the proofs of Claim \ref{cl:ChiUpperBound},
		Corollary \ref{cor:chi0mod3} and Claim \ref{cl:chi2mod3},  the lower
		bound for $\chi_{r,i}^{}$ (for $i = 1,0,2$, respectively) was in fact
		given by a clique construction. Therefore, for every $r\in\cA$, the
		edges of $K_{\chi^{}_r}$ can be covered by $r$ factors.
	\end{remark}
	


	To conclude this section, we consider two of the values of $r$ which are
	not covered by Theorem~\ref{thm:union}, namely, $r = 3$ and $r = 6$.
	
	It is quite easy to cover the edges of $K_5$ with three factors (we
	omit the straightforward details), and thus it follows by
	Claim~\ref{cl:ChiUpperBound} that $5 \le \chi_3^{} \le 6$. We do
	not know which of these two bounds is tight, but let us point
	out that if Conjecture~\ref{conj:MolloyReed} is true, then
	$\chi_3^{} = 5$ holds by Claim~\ref{cl:356}, thus providing an
	affirmative answer to Question~\ref{qst:union3}.
	
	Finally, $\chi_6^{} \le 12$ by Claim~\ref{cl:ChiUpperBound}.
	Although we know by Theorem~\ref{thm:0mod6} that the edges of
	$K_{12}$ cannot be covered by six proper factors, this does not imply,
	of course, that $\chi_6^{} < 12$.
	We give the following construction to show that $K_{11}$ can be
	covered by six factors, implying that $11 \le \chi_6^{} \le 12$.
	Here, too, if Conjecture~\ref{conj:MolloyReed} is true, then
	$\chi_6^{} = 11$ holds by Claim~\ref{cl:356}.
	
	\iffigure
	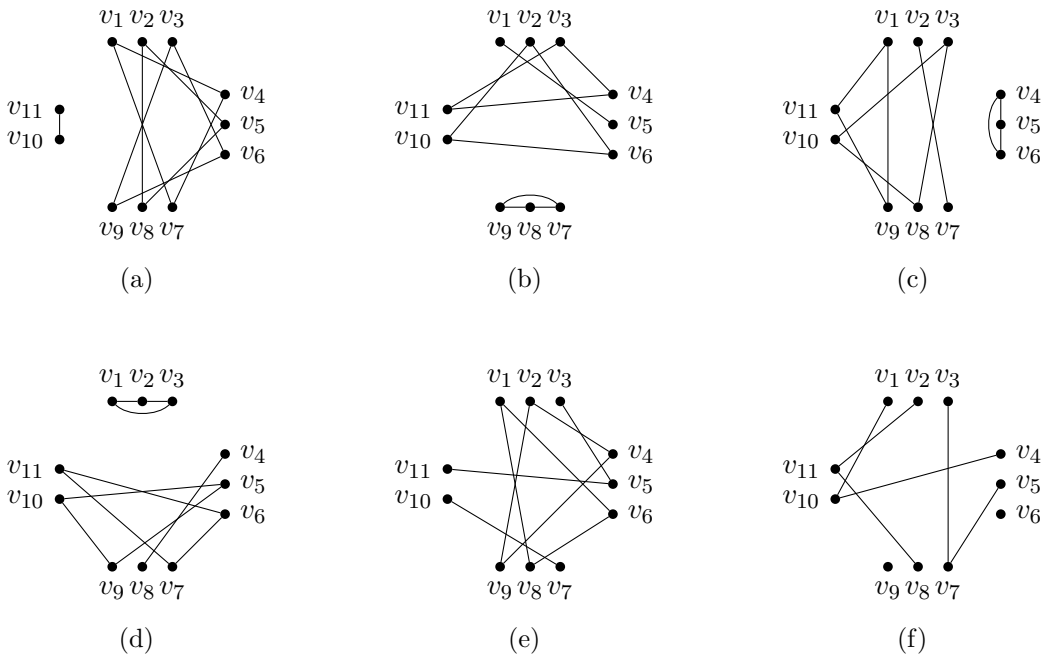
\begin{figure}[H]
		\centering
		
		\begin{subfigure}[t]{0.3\textwidth}
			\centering
			
			\begin{tikzpicture}[auto,
			vertex/.style={circle,draw=black!100,fill=black!100, thick,
				inner sep=0pt,minimum size=1mm}]
			\node (v1) at (1,1.5) [vertex,label=above:$v_1$] {};
			\node (v2) at (1.4,1.5) [vertex,label=above:$v_2$] {};
			\node (v3) at (1.8,1.5) [vertex,label=above:$v_3$] {};
			
			\node (v4) at (2.5,0.8) [vertex,label=right:$v_4$] {};
			\node (v5) at (2.5,0.4) [vertex,label=right:$v_5$] {};
			\node (v6) at (2.5,0) [vertex,label=right:$v_6$] {};
			
			\node (v7) at (1.8,-0.7) [vertex,label=below:$v_7$] {};
			\node (v8) at (1.4,-0.7) [vertex,label=below:$v_8$] {};
			\node (v9) at (1,-0.7) [vertex,label=below:$v_9$] {};
			
			\node (v10) at (0.3,0.2) [vertex,label=left:$v_{10}$] {};
			\node (v11) at (0.3,0.6) [vertex,label=left:$v_{11}$] {};

			\draw [-] (v1) --node[inner sep=0pt,swap]{} (v4);
			\draw [-] (v1) --node[inner sep=0pt,swap]{} (v7);
			\draw [-] (v7) --node[inner sep=0pt,swap]{} (v4);
			
			\draw [-] (v2) --node[inner sep=0pt,swap]{} (v5);
			\draw [-] (v2) --node[inner sep=0pt,swap]{} (v8);
			\draw [-] (v8) --node[inner sep=0pt,swap]{} (v5);
			
			\draw [-] (v3) --node[inner sep=0pt,swap]{} (v6);
			\draw [-] (v3) --node[inner sep=0pt,swap]{} (v9);
			\draw [-] (v9) --node[inner sep=0pt,swap]{} (v6);
			
			\draw [-] (v10) --node[inner sep=0pt,swap]{} (v11);
			
			\end{tikzpicture}
			
			\caption{}
			
		\end{subfigure}
		\begin{subfigure}[t]{0.3\textwidth}
			\centering
			
			\begin{tikzpicture}[auto,
			vertex/.style={circle,draw=black!100,fill=black!100, thick,
				inner sep=0pt,minimum size=1mm}]
			\node (v1) at (1,1.5) [vertex,label=above:$v_1$] {};
			\node (v2) at (1.4,1.5) [vertex,label=above:$v_2$] {};
			\node (v3) at (1.8,1.5) [vertex,label=above:$v_3$] {};
			
			\node (v4) at (2.5,0.8) [vertex,label=right:$v_4$] {};
			\node (v5) at (2.5,0.4) [vertex,label=right:$v_5$] {};
			\node (v6) at (2.5,0) [vertex,label=right:$v_6$] {};
			
			\node (v7) at (1.8,-0.7) [vertex,label=below:$v_7$] {};
			\node (v8) at (1.4,-0.7) [vertex,label=below:$v_8$] {};
			\node (v9) at (1,-0.7) [vertex,label=below:$v_9$] {};
			
			\node (v10) at (0.3,0.2) [vertex,label=left:$v_{10}$] {};
			\node (v11) at (0.3,0.6) [vertex,label=left:$v_{11}$] {};

			\draw [-] (v2) --node[inner sep=0pt,swap]{} (v6);
			\draw [-] (v2) --node[inner sep=0pt,swap]{} (v10);
			\draw [-] (v10) --node[inner sep=0pt,swap]{} (v6);
			
			\draw [-] (v3) --node[inner sep=0pt,swap]{} (v4);
			\draw [-] (v3) --node[inner sep=0pt,swap]{} (v11);
			\draw [-] (v11) --node[inner sep=0pt,swap]{} (v4);
			
			\draw [-] (v7) --node[inner sep=0pt,swap]{} (v8);
			\draw [-] (v7) .. controls (1.6,-0.5) and (1.2,-0.5) ..
			(v9);
			\draw [-] (v9) --node[inner sep=0pt,swap]{} (v8);
			
			\draw [-] (v1) --node[inner sep=0pt,swap]{} (v5);
			\end{tikzpicture}
			\caption{}
			
		\end{subfigure}
		\begin{subfigure}[t]{0.3\textwidth}
			\centering
			
			\begin{tikzpicture}[auto,
			vertex/.style={circle,draw=black!100,fill=black!100, thick,
				inner sep=0pt,minimum size=1mm}]
			\node (v1) at (1,1.5) [vertex,label=above:$v_1$] {};
			\node (v2) at (1.4,1.5) [vertex,label=above:$v_2$] {};
			\node (v3) at (1.8,1.5) [vertex,label=above:$v_3$] {};
			
			\node (v4) at (2.5,0.8) [vertex,label=right:$v_4$] {};
			\node (v5) at (2.5,0.4) [vertex,label=right:$v_5$] {};
			\node (v6) at (2.5,0) [vertex,label=right:$v_6$] {};
			
			\node (v7) at (1.8,-0.7) [vertex,label=below:$v_7$] {};
			\node (v8) at (1.4,-0.7) [vertex,label=below:$v_8$] {};
			\node (v9) at (1,-0.7) [vertex,label=below:$v_9$] {};
			
			\node (v10) at (0.3,0.2) [vertex,label=left:$v_{10}$] {};
			\node (v11) at (0.3,0.6) [vertex,label=left:$v_{11}$] {};

			\draw [-] (v1) --node[inner sep=0pt,swap]{} (v9);
			\draw [-] (v1) --node[inner sep=0pt,swap]{} (v11);
			\draw [-] (v11) --node[inner sep=0pt,swap]{} (v9);
			
			\draw [-] (v3) --node[inner sep=0pt,swap]{} (v8);
			\draw [-] (v3) --node[inner sep=0pt,swap]{} (v10);
			\draw [-] (v10) --node[inner sep=0pt,swap]{} (v8);
			
			\draw [-] (v4) --node[inner sep=0pt,swap]{} (v5);
			\draw [-] (v4) .. controls (2.3,0.6) and (2.3,0.2) .. (v6);
			\draw [-] (v6) --node[inner sep=0pt,swap]{} (v5);
			
			\draw [-] (v2) --node[inner sep=0pt,swap]{} (v7);
			\end{tikzpicture}
			\caption{}

		\end{subfigure}

		\vspace*{8mm}
		\begin{subfigure}[t]{0.3\textwidth}
			\centering
			
			\begin{tikzpicture}[auto,
			vertex/.style={circle,draw=black!100,fill=black!100, thick,
				inner sep=0pt,minimum size=1mm}]
			\node (v1) at (1,1.5) [vertex,label=above:$v_1$] {};
			\node (v2) at (1.4,1.5) [vertex,label=above:$v_2$] {};
			\node (v3) at (1.8,1.5) [vertex,label=above:$v_3$] {};
			
			\node (v4) at (2.5,0.8) [vertex,label=right:$v_4$] {};
			\node (v5) at (2.5,0.4) [vertex,label=right:$v_5$] {};
			\node (v6) at (2.5,0) [vertex,label=right:$v_6$] {};
			
			\node (v7) at (1.8,-0.7) [vertex,label=below:$v_7$] {};
			\node (v8) at (1.4,-0.7) [vertex,label=below:$v_8$] {};
			\node (v9) at (1,-0.7) [vertex,label=below:$v_9$] {};
			
			\node (v10) at (0.3,0.2) [vertex,label=left:$v_{10}$] {};
			\node (v11) at (0.3,0.6) [vertex,label=left:$v_{11}$] {};

			\draw [-] (v5) --node[inner sep=0pt,swap]{} (v9);
			\draw [-] (v5) --node[inner sep=0pt,swap]{} (v10);
			\draw [-] (v10) --node[inner sep=0pt,swap]{} (v9);
			
			\draw [-] (v6) --node[inner sep=0pt,swap]{} (v7);
			\draw [-] (v6) --node[inner sep=0pt,swap]{} (v11);
			\draw [-] (v11) --node[inner sep=0pt,swap]{} (v7);
			
			\draw [-] (v1) --node[inner sep=0pt,swap]{} (v2);
			\draw [-] (v1) .. controls (1.2,1.3) and (1.6,1.3) .. (v3);
			\draw [-] (v3) --node[inner sep=0pt,swap]{} (v2);
			
			\draw [-] (v4) --node[inner sep=0pt,swap]{} (v8);
			\end{tikzpicture}
			\caption{}
			
		\end{subfigure}
		\begin{subfigure}[t]{0.3\textwidth}
			\centering
			
			\begin{tikzpicture}[auto,
			vertex/.style={circle,draw=black!100,fill=black!100, thick,
				inner sep=0pt,minimum size=1mm}]
			\node (v1) at (1,1.5) [vertex,label=above:$v_1$] {};
			\node (v2) at (1.4,1.5) [vertex,label=above:$v_2$] {};
			\node (v3) at (1.8,1.5) [vertex,label=above:$v_3$] {};
			
			\node (v4) at (2.5,0.8) [vertex,label=right:$v_4$] {};
			\node (v5) at (2.5,0.4) [vertex,label=right:$v_5$] {};
			\node (v6) at (2.5,0) [vertex,label=right:$v_6$] {};
			
			\node (v7) at (1.8,-0.7) [vertex,label=below:$v_7$] {};
			\node (v8) at (1.4,-0.7) [vertex,label=below:$v_8$] {};
			\node (v9) at (1,-0.7) [vertex,label=below:$v_9$] {};
			
			\node (v10) at (0.3,0.2) [vertex,label=left:$v_{10}$] {};
			\node (v11) at (0.3,0.6) [vertex,label=left:$v_{11}$] {};

			\draw [-] (v1) --node[inner sep=0pt,swap]{} (v6);
			\draw [-] (v1) --node[inner sep=0pt,swap]{} (v8);
			\draw [-] (v8) --node[inner sep=0pt,swap]{} (v6);
			
			\draw [-] (v2) --node[inner sep=0pt,swap]{} (v4);
			\draw [-] (v2) --node[inner sep=0pt,swap]{} (v9);
			\draw [-] (v9) --node[inner sep=0pt,swap]{} (v4);
			
			\draw [-] (v3) --node[inner sep=0pt,swap]{} (v5);
			\draw [-] (v11) --node[inner sep=0pt,swap]{} (v5);
			
			\draw [-] (v7) --node[inner sep=0pt,swap]{} (v10);
			\end{tikzpicture}
			\caption{}
			
		\end{subfigure}
		\begin{subfigure}[t]{0.3\textwidth}
			\centering
			
			\begin{tikzpicture}[auto,
			vertex/.style={circle,draw=black!100,fill=black!100, thick,
				inner sep=0pt,minimum size=1mm}]
			\node (v1) at (1,1.5) [vertex,label=above:$v_1$] {};
			\node (v2) at (1.4,1.5) [vertex,label=above:$v_2$] {};
			\node (v3) at (1.8,1.5) [vertex,label=above:$v_3$] {};
			
			\node (v4) at (2.5,0.8) [vertex,label=right:$v_4$] {};
			\node (v5) at (2.5,0.4) [vertex,label=right:$v_5$] {};
			\node (v6) at (2.5,0) [vertex,label=right:$v_6$] {};
			
			\node (v7) at (1.8,-0.7) [vertex,label=below:$v_7$] {};
			\node (v8) at (1.4,-0.7) [vertex,label=below:$v_8$] {};
			\node (v9) at (1,-0.7) [vertex,label=below:$v_9$] {};
			
			\node (v10) at (0.3,0.2) [vertex,label=left:$v_{10}$] {};
			\node (v11) at (0.3,0.6) [vertex,label=left:$v_{11}$] {};

			\draw [-] (v1) --node[inner sep=0pt,swap]{} (v10);
			\draw [-] (v4) --node[inner sep=0pt,swap]{} (v10);
			
			\draw [-] (v2) --node[inner sep=0pt,swap]{} (v11);
			\draw [-] (v8) --node[inner sep=0pt,swap]{} (v11);
			
			\draw [-] (v3) --node[inner sep=0pt,swap]{} (v7);
			\draw [-] (v5) --node[inner sep=0pt,swap]{} (v7);
			
			\end{tikzpicture}
			\caption{}
			
		\end{subfigure}
		
		\caption{Covering $K_{11}$}
		
	\end{figure}
	
	\fi
	
	\section{Ramsey niceness}
	\label{sec:nice}
	
	In this section we prove Theorems~\ref{thm:main},~\ref{thm:F7},
	and~\ref{thm:NiceStar}. For convenience, we view niceness in a
	slightly different way. Let $\cF$ be a family of fixed graphs.
	\begin{definition}
		Let $\clq_k(\cF)$ be the maximum integer $s$ such that there
		exists a $k$-coloring  of $E(K_s)$ with no monochromatic copy of any
		$F \in \cF$.
	\end{definition}
	
	Note that $\clq_k(\cF) = R_k(\cF) - 1$ for every family $\cF$ and
	every integer $k$.
	
	\begin{definition}
		Let $\g_k^{}(\cF)$ be the maximum integer $s$ such that
		there exists a graph $G$ with $\chi(G)=s$ and a $k$-coloring of
		$E(G)$ with
		no monochromatic copy of any $F \in \cF $.
	\end{definition}
	
	Note that a family $\cF$ is $k$-nice if and only if
	$\clq_k^{}(\cF)=\g_k^{}(\cF)$, and that, by definition,
	$\clq_k^{}(\cF)\leq \g_k^{}(\cF)$. Hence, in order to
	prove that a family $\cF$ is $k$-nice, it suffices to
	prove that $\clq_k^{}(\cF) \geq \g_k^{}(\cF)$.
	
	Recall our notation for families of connected graphs consisting of 3
	edges, namely $\cF_3=\{S_3\}$, $\cF_4=\{K_3,P_4\}$,
	$\cF_5=\{K_3,S_3\}$, $\cF_6=\{P_4,S_3\}$, and
	$\cF_7=\{K_3,P_4,S_3\}$. We prove the niceness of these families (as
	stated in Theorems~\ref{thm:main} and~\ref{thm:F7}) in four separate
	claims.
	
	\begin{claim}
		$\clq_k^{}(\cF_6)=\g_k^{}(\cF_6)=\chi^{}_k$ for every $k \in \cA$,
		where $\chi^{}_k$ is as defined in Section~\ref{sec:trianglefactors}.
		In particular, $\cF_6$ is $k$-nice for every $k\in \cA$.
	\end{claim}
	
	\begin{proof}
		Let $k \in \cA$ and let $G$ be a graph such that there exists a
		$k$-coloring of its edges with no monochromatic copy of any element
		of $\cF_6$. Every color class in such a coloring must be a
		generalized triangle factor and thus by definition $\g_k^{}(\cF_6)=
		\chi_k^{}$. By Remark \ref{rem:LowerBoundClique} we get
		$\clq_k^{}(\cF_6)=\g^{}_k(\cF_6)=\chi^{}_k$.
	\end{proof}
	
	\begin{claim}
		$\clq_k^{}(\cF_3)=\g_k^{}(\cF_3)=\clq_k^{}(\cF_5)=\g_k^{}(\cF_5)=2k+1$
		for every positive integer $k$ with the only exception
		$\clq_1(\cF_5)=2$. In particular, $\cF_3$ is $k$-nice for every
		$k\geq 1$ and $\cF_5$ is $k$-nice for every $k\geq 2$.
	\end{claim}
	
	\begin{proof}
		The fact that $\cF_3$ is $k$-nice for every $k$ follows from the
		result in~\cite{GoodGraph} that every star is $k$-nice for every
		$k$. For the sake of completness, we include a simple proof. Let $G$
		be a graph such that there exists a $k$-coloring of its edges with
		no monochromatic copy of $S_3$. Then the maximum degree in every
		color class is at most $2$. Therefore $\Delta(G) \leq 2k$, and thus
		$\g_k^{}(\cF_3), \g_k^{}(\cF_5) \leq 2k+1$ for every $k$. On the
		other hand, in 1890 Walecki showed a decomposition of the complete
		graph $K_{2k+1}$ into $k$ Hamilton cycles for every $k$ (see,
		e.g.,~\cite{WaleckiCons,WaleckiCons1}). By assigning each Hamilton cycle a
		different color we conclude that $\clq_k^{}(\cF_3)\geq 2k+1$ for
		every $k\geq 1$ and $\clq_k^{}(\cF_5)\geq 2k+1$ for every $k\geq 2$
		(note that the latter argument does not apply to $\clq_1^{}(\cF_5)$
		since in this case $C_n = K_3 \in \cF_5$).
		
		
		It remains to deal with $\clq_1^{}(\cF_5)$ and $\g_1^{}(\cF_5)$.
		Clearly, the maximal clique containing no triangles is $K_2$,
		therefore $\clq_1^{}(\cF_5)=2$. For a general graph, $C_5$ is an
		example of an $\cF_5$-free graph with chromatic number 3, thus
		showing $\g_1^{}(\cF_5) \ge 3$.
	\end{proof}

	\begin{claim}
		$\clq_1^{}(\cF_4)=\g_1^{}(\cF_4)=2$,
		$\clq_2^{}(\cF_4)=\g_2^{}(\cF_4)=3$, and
		$\clq_k^{}(\cF_4)=\g_k^{}(\cF_4)=2k-2$ for every $k\geq 3$.
		In particular, $\cF_4$ is $k$-nice for every $k$.
	\end{claim}
	
	\begin{proof}
		Call a graph in which every connected component is a star a
		\emph{galaxy}. Note that if an edge coloring of an arbitrary graph
		contains no monochromatic copy of any element of $\cF_4$, then every
		color
		class in this coloring must be a galaxy.
		
		We first deal with the cases $k = 1,2$. Since every galaxy is
		2-colorable,
		and $K_2$ is obviously a galaxy, it follows that
		$\clq_1^{}(\cF_4)=\g_1^{}(\cF_4)=2$. The edges of $K_3$ can be
		trivially covered by two galaxies, thus $\clq_2^{}(\cF_4) \ge 3$.
		Now
		let $G$ be a union of two galaxies. We show that $G$ is 3-colorable,
		thus proving $\clq_2^{}(\cF_4)=\g_2^{}(\cF_4)=3$. By
		Observation~\ref{obs:coloring} it suffices to show that the 3-core
		of $G$ is empty and therefore trivially 3-colorable. Assume for a
		contradiction that $G$ has a non empty 3-core $K = (V, E)$ and
		denote $n := |V|$, $e := |E|$. For $i=1,2$, let $G_i$ be the
		intersection of the $i$th galaxy and $K$, and note that each $G_i$
		is a galaxy as well, since it is a subgraph of a galaxy. Let $e_i :=
		|E(G_i)|$ and let $s_i$ be the number of stars in $G_i$. Note that
		$s_i + e_i = n$ for every $i$. Since $\delta(K) \ge 3$, every
		vertex $v \in
		V$ must have degree at least $2$ in at least one of the
		galaxies. Hence, every $v \in V$ is the center vertex of a non
		trivial star (i.e., a star with at least two edges) in at least one
		galaxy,
		implying that $s_1 + s_2 \geq n$. This in turn implies that
		$e \le e_1 + e_2 = (n-s_1) + (n-s_2) \le n$, and thus the
		average degree in $K$
		is at most $2$, a contradiction.
		
		We now consider the general case $k \ge 3$. We show that any graph
		which is the union of $k \geq 3$ galaxies is $(2k-2)$-colorable,
		thus
		proving that $\g_k^{}(\cF_4)\leq 2k-2$. For better readability, we
		show
		that for every $k \ge 2$, the union of $k+1$ galaxies is
		$2k$-colorable.
		
		Let $G$ be the union of $k+1$ galaxies on the same set of vertices,
		let $K = (V,E)$ be the $2k$-core of $G$ and let $n = |V|$,
		$e = |E|$. By Observation~\ref{obs:coloring} it suffices to show
		that
		$K$ is $2k$-colorable. If $K$ is empty, then we are clearly done;
		assume then
		that it is not. As in the union of two galaxies, for every $1 \le i
		\le
		k+1$, let $G_i$ be the galaxy obtained by the intersection of $K$
		and the $i$th galaxy on $G$, let $e_i := |E(G_i)|$ and let $s_i$ be
		the number of stars in $G_i$, so $s_i + e_i = n$ for every $i$.
		
		Since $\delta(K) \ge 2k > k+1$, once again every vertex $v \in V$
		must have degree at least $2$ in at least one galaxy and
		therefore must be the center vertex of a non trivial star in at
		least one galaxy.
		It follows that $\sum_{i=1}^{k+1}s_i\geq n$. Moreover, $\delta(K)
		\ge 2k$ implies that $e\geq kn$. Putting it all together we get
		
		$$kn \le e \leq
		\sum_{i=1}^{k+1}e_i=\sum_{i=1}^{k+1}(n-s_i)=(k+1)n-\sum_{i=1}^{k+1}s_i\leq
		kn.$$
		Hence, all the inequalities above are in fact equalities, and in particular we have
        $\sum_{i=1}^{k+1}s_i=n$, which implies that there are exactly $n$
			stars in total, none of them is trivial, and every vertex of $K$ is the center of
			exactly one star. Now note that if two vertices are star centers in the same galaxy, then they are not connected by an edge in that galaxy, and not in any other galaxy (since none of them is a star center in other galaxies). Therefore, assigning the color $i$ for the star centers of the $i$th galaxy, for every $i \in [k+1]$, yields a legal coloring of $K$ with $k+1 < 2k$ colors (recall that every vertex is a star center in some galaxy). The upper bound $\g_k^{}(\cF_4)\leq 2k-2$ is thus established.

		For the lower bound, we now show how the edges of $K_{2k}$ can be
		covered with $k+1$ galaxies for every $k \ge 2$, hence
		$\clq_k^{}(\cF)\geq 2k-2$ for every $k \ge 3$. Denote the vertices
		of
		the clique by $\{v_1,\ldots,v_{2k}\}$. For every $ i\in [k]$ let the
		$i$th galaxy consist of the following two stars: the first center
		is $v_i$ with $v_{i+1},\ldots,v_{i+k-1}$ as leaves, and the second
		center is $v_{i+k}$ with $v_{i+k+1},\ldots,v_{i+2k-1}$ as leaves,
		where all indices are calculated modulo $2k$ (with the one
		exception of $v_{2k}$ not being referred to as $v_0$). The last
		galaxy is the
		following perfect matching: $\left\{ v_iv_{i+k} \mid i\in [k]
		\right\}$. See Figure~\ref{fig:K2k}.
		
		\iffigure

		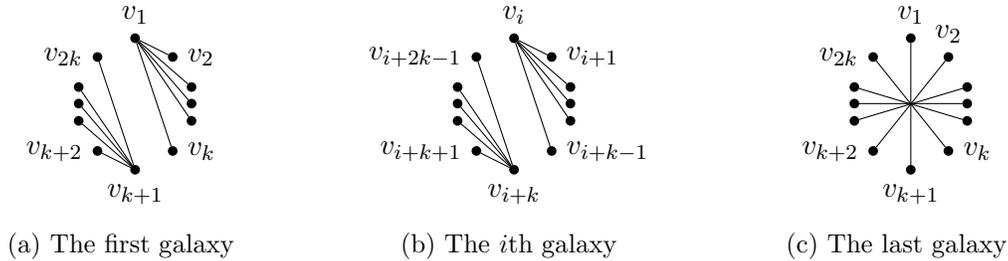
\begin{figure}[H]
			\centering
			
			\begin{subfigure}[t]{0.3\textwidth}

				\centering
				\begin{tikzpicture}[auto,
				vertex/.style={circle,draw=black!100,fill=black!100, thick,
					inner sep=0pt,minimum size=1mm}]
				\node (v1) at (1,1) [vertex,label=above:$v_1$] {};
				\node (v2) at (1.5,0.75) [vertex,label=right:$v_2$] {};
				\node (1) at (1.75,0.35) [vertex] {};
				\node (2) at (1.75,0.13) [vertex] {};
				\node (3) at (1.75,-0.1) [vertex] {};
				\node (vk) at (1.5,-0.5) [vertex,label=right:$v_{k}$] {};
				\node (vk+1) at (1,-0.75) [vertex,label=below:$v_{k+1}$] {};
				\node (vk+2) at (0.5,-0.5) [vertex,label=left:$v_{k+2}$] {};
				\node (4) at (0.25,-0.1) [vertex] {};
				\node (5) at (0.25,0.13) [vertex] {};
				\node (6) at (0.25,0.35) [vertex] {};
				\node (v2k) at (0.5,0.75) [vertex,label=left:$v_{2k}$] {};
				
				\draw [-] (v1) --node[inner sep=0pt,swap]{} (v2);
				\draw [-] (v1) --node[inner sep=0pt,swap]{} (1);
				\draw [-] (v1) --node[inner sep=0pt,swap]{} (2);
				\draw [-] (v1) --node[inner sep=0pt,swap]{} (3);
				\draw [-] (v1) --node[inner sep=0pt,swap]{} (vk);
				
				\draw [-] (vk+1) --node[inner sep=0pt,swap]{} (vk+2);
				\draw [-] (vk+1) --node[inner sep=0pt,swap]{} (4);
				\draw [-] (vk+1) --node[inner sep=0pt,swap]{} (5);
				\draw [-] (vk+1) --node[inner sep=0pt,swap]{} (6);
				\draw [-] (vk+1) --node[inner sep=0pt,swap]{} (v2k);
				
				\end{tikzpicture}
				\caption{The first galaxy}
			\end{subfigure}
			\begin{subfigure}[t]{0.3\textwidth}

				\centering
				\begin{tikzpicture}[auto,
				vertex/.style={circle,draw=black!100,fill=black!100, thick,
					inner sep=0pt,minimum size=1mm}]
				\node (vi) at (1,1) [vertex,label=above:$v_i$] {};
				\node (vi+1) at (1.5,0.75) [vertex,label=right:$v_{i+1}$]
				{};
				\node (1) at (1.75,0.35) [vertex] {};
				\node (2) at (1.75,0.13) [vertex] {};
				\node (3) at (1.75,-0.1) [vertex] {};
				\node (vi+k-1) at (1.5,-0.5)
				[vertex,label=right:$v_{i+k-1}$] {};
				\node (vi+k) at (1,-0.75) [vertex,label=below:$v_{i+k}$] {};
				\node (vi+k+1) at (0.5,-0.5) [vertex,label=left:$v_{i+k+1}$]
				{};
				\node (4) at (0.25,-0.1) [vertex] {};
				\node (5) at (0.25,0.13) [vertex] {};
				\node (6) at (0.25,0.35) [vertex] {};
				\node (vi+2k-1) at (0.5,0.75)
				[vertex,label=left:$v_{i+2k-1}$] {};
				
				\draw [-] (v1) --node[inner sep=0pt,swap]{} (v2);
				\draw [-] (v1) --node[inner sep=0pt,swap]{} (1);
				\draw [-] (v1) --node[inner sep=0pt,swap]{} (2);
				\draw [-] (v1) --node[inner sep=0pt,swap]{} (3);
				\draw [-] (v1) --node[inner sep=0pt,swap]{} (vk);
				
				\draw [-] (vk+1) --node[inner sep=0pt,swap]{} (vk+2);
				\draw [-] (vk+1) --node[inner sep=0pt,swap]{} (4);
				\draw [-] (vk+1) --node[inner sep=0pt,swap]{} (5);
				\draw [-] (vk+1) --node[inner sep=0pt,swap]{} (6);
				\draw [-] (vk+1) --node[inner sep=0pt,swap]{} (v2k);
				
				\end{tikzpicture}
				\caption{The $i$th galaxy}
			\end{subfigure}
			\begin{subfigure}[t]{0.3\textwidth}

				\centering
				\begin{tikzpicture}[auto,
				vertex/.style={circle,draw=black!100,fill=black!100, thick,
					inner sep=0pt,minimum size=1mm}]
				\node (v1) at (1,1) [vertex,label=above:$v_1$] {};
				\node (v2) at (1.5,0.75) [vertex,label=above:$v_2$] {};
				\node (1) at (1.75,0.35) [vertex] {};
				\node (2) at (1.75,0.13) [vertex] {};
				\node (3) at (1.75,-0.1) [vertex] {};
				\node (vk) at (1.5,-0.5) [vertex,label=right:$v_{k}$] {};
				\node (vk+1) at (1,-0.75) [vertex,label=below:$v_{k+1}$] {};
				\node (vk+2) at (0.5,-0.5) [vertex,label=left:$v_{k+2}$] {};
				\node (4) at (0.25,-0.1) [vertex] {};
				\node (5) at (0.25,0.13) [vertex] {};
				\node (6) at (0.25,0.35) [vertex] {};
				\node (v2k) at (0.5,0.75) [vertex,label=left:$v_{2k}$] {};
				
				\draw [-] (v1) --node[inner sep=0pt,swap]{} (vk+1);
				\draw [-] (v2) --node[inner sep=0pt,swap]{} (vk+2);
				\draw [-] (1) --node[inner sep=0pt,swap]{} (4);
				\draw [-] (2) --node[inner sep=0pt,swap]{} (5);
				\draw [-] (3) --node[inner sep=0pt,swap]{} (6);
				\draw [-] (vk) --node[inner sep=0pt,swap]{} (v2k);
				
				\end{tikzpicture}
				
				\caption{The last galaxy}
				
			\end{subfigure}
			
			\caption{Covering $K_{2k}$}
			\label{fig:K2k}
		\end{figure}
		
		\fi
		
	\end{proof}

	\begin{claim}
		There exists a positive integer $k_0$ such that
		$\clq_k^{}(\cF_7)=\g_k^{}(\cF_7)=\frac {4k}3+1$ for every $k \geq
		k_0$ satisfying $k \md{6}{9}$. In particular, $\cF_7$ is $k$-nice
		for infinitely many values of $k$.
	\end{claim}
	
	\begin{proof}
		Note that if $G$ is a graph containing no member of $\cF_7$ as a
		subgraph, then every connected component of $G$ is a path of length
		at most two. Denote such a graph by \textit{SPG} ({Short Paths
			Graph}) and denote a graph obtained by a union of $k$ such graphs by
		$kSPG$.
		
		Let $s = \g_k(\cF_7)$ and let $G$ be a $kSPG$ such that $|V(G)| =
		n$ and $G$ is $s$-critical. In particular, the minimal degree in $G$
		is at least $s-1$, hence $|E(G)| \ge \frac{n(s-1)}{2}$. However,
		every $SPG$ on $n$ vertices contains at most $2n/3$ edges, therefore
		$|E(G)| \le 2kn/3$. Putting it together we get $s \le \frac {4k}{3}
		+ 1$, and in particular whenever $k = 9r +6$ (for some positive
		integer $r$) we get $s \le 12r +9$.
		
		Observe that the graph $ P_3$ has a degree sequence $2,1,1$ and
		average degree $4/3$, and the smallest integer $\gamma$ such that
		$\left(\gamma, 3\gamma / 4\right) \in
		span^{}_\mathbb{Z}\left\{(2,1),(1,1)\right\}$ is 4. By
		Theorem~\ref{thm:Hdecomp}, for every sufficiently large $n$
		satisfying $n \md{9}{12}$, there exists a decomposition of $K_n$
		into $P_3$-factors, i.e., $K_n$ is a  $kSPG$ for those values of
		$n$, for $k = \frac{3(n-1)}{4}$. In other words, for any
		sufficiently  large $k$ of the form $k=9r+6$, $K_{12r+9}$ is a
		$kSPG$, thus the proof is complete.
	\end{proof}

	We conclude this section with the proof of
	Theorem~\ref{thm:NiceStar}.
	
	\begin{claim}
		Let $r$ be a positive integer, and let $\cF$ be a family of graphs
		such that $K_{1,r+1} \in \cF$, and all other $F \in \cF$ contain
		at least one cycle.  Then, $\clq_k^{}(\cF)=\g_k^{}(\cF)=kr+1$  for
		infinitely many values of $k$, and, in particular,  $\cF$ is
		$k$-nice for those values of $k$.
	\end{claim}

	\begin{proof}
		It is easy to see that $\g_k^{}(\cF) \le kr + 1$. Indeed, assume for
		contradiction that $\g_k(\cF) \ge kr + 2$, and let  $G$ be a graph such that $\chi(G) =
		\g_k(\cF)$ and there exists a $k$-coloring of
		$E(G)$ with
		no monochromatic copy of any $F \in \cF $. Then by Brooks' theorem $\Delta(G) \ge kr +
		1$, and thus in every $k$-coloring of $E(G)$ there must be a color
		class in which $r+1$ edges intersect in the same vertex, and a
		forbidden copy of $K_{1,r+1}$ appears.
		
		Let $g$ be the largest girth among all members $F \in \cF \stm
		\{K_{1,r+1}\}$, and let $H$ be an $r$-regular graph with girth
		greater than $g$, such that $|V(H)| = h$ for some integer $h$
		relatively prime to $r$ (it is well known that such graphs exist, it
		follows for example from the technique of Erd\H{o}s and
		Sachs~\cite{ES}). By Theorem~\ref{thm:Hdecomp}, there exists a
		decomposition of $K_n$ into $H$-factors for every sufficiently large
		$n$ satisfying $n \md {0}{h}$ and $n-1 \md {0}{r}$ (here, in the
		terminology of Theorem~\ref{thm:Hdecomp}, as $H$ is $r$-regular, we
		are looking for the smallest integer $\gamma$ such that
		$\left(\gamma, \gamma / r\right) \in
		span^{}_\mathbb{Z}\left\{(r,1)\right\}$, which is clearly $r$). By
		the Chinese Remainder Theorem,
		there exist
		infinitely many values of $n$
		satisfying the congruences\textbf{}
		modulo $hr$, since $h$ and $r$ are relatively prime.
		For any such $n$, let  $k=\frac{n-1}{r}$.
		Then $K_{kr+1}$ can be decomposed into $k$ $H$-factors, hence $c_k(F) \ge kr+1$.
	\end{proof}

	\section{Matchings in hypergraphs}\label{sec:matchings}
	
	In this section we disprove Conjecture~\ref{conj:Aharoni} in a
	strong sense.\\
	
	Let $d$ be a positive integer and let $m = \lfloor 3d/2 \rfloor$, $A =
	[d]$
	and $B = [m] \setminus [d]$. We define an $m\times m\times m$
	3-uniform, 3-partite,
	$d$-regular simple (containing no repeated edges) hypergraph $\cH =
	(V,E)$ as follows.
	Let $V = X\cup Y\cup Z$, where $X=\{x_1,\ldots,x_m\}$,
	$Y=\{y_1,\ldots,y_m\}$ and $Z=\{z_1,\ldots,z_m\}$. For every $i \in A$
	and every $j \in B$ we add the three hyperedges $\{x_i,y_i,z_j\},
	\{x_i,y_j,z_i\},\{x_j,y_i,z_i\}$ to $E$. If $d$ is odd we also add
	to $E$ all the hyperedges $\{x_i,y_i,z_i\}$ for every $i \in A$.
	Note that $\cH$ is indeed 3-partite and $d$-regular. Let $f: E
	\rightarrow A$ be the function which assigns each $e \in E$ the
	unique index $i \in A$ such that $|e\cap \{x_i,y_i,z_i\}| \ge 2$.
	Since every two edges $e_1, e_2 \in E$ such that $f(e_1) = f(e_2)$
	intersect, every matching $M$ in $\cH$ contains at most one edge
	from the set $f^{-1}(i)$ for every $i \in A$, implying that
	$|M| \le |A| = d$.
	
	By taking disjoint copies of $\cH$ we obtain a $d$-regular $n\times
	n\times n$ 3-partite 3-uniform hypergraph (for arbitrarily large $n$)
	such
	that the size of the maximum matching in the hypergraph is $d \cdot n/m
	= 2n/3$ for even $d$, and $d \cdot n/m = \frac{2d}{3d-1} \cdot n$
	for odd $d$, thus disproving Conjecture~\ref{conj:Aharoni} for every
	$d \ge 4$.
	
	In~\cite{AK} it is conjectured that the edges of any $r$-uniform
	hypergraph with maximum degree $d$
	in which every pair of edges share at most $t$ common
	vertices can be covered by $(t-1+1/t+o(1))d$ matchings, where
	the $o(1)$ term tends to zero as $d$ tends to infinity. If true,
	then, by taking $r=3$ and $t=2$, this shows that the constant $2/3$
	above is tight for large $d$.
	It is also worth noting that the known results
	about nearly perfect matchings in regular linear hypergraphs
	imply that every $d$-regular $3$-uniform linear hypergraph on $n$
	vertices
	contains a nearly perfect  matching. In particular, it is shown
	in~\cite{AKS},
	using the R\"odl nibble, that any such hypergraph
	contains a matching covering
	all but at most $O(n \ln^{3/2}d/\sqrt d)$ vertices.\\
	
	In relation to the example given above, we note the following. It is easy to see that every $r$-uniform $d$ regular hypergraph on
	$n$ vertices contains a matching that covers at least
	$$
	\frac{nd}{1+(d-1)r} >\frac{n}{r}
	$$
	vertices. Indeed, the number of edges is $e=nd/r$, and each edge
	intersects at most $r(d-1)$ others, hence the line-graph, whose
	vertices are the edges of the hypergraph where two are adjacent if and only if
	they intersect, has an independent  set of size at least
	$\frac{e}{1+r(d-1)}$, implying the above estimate.
	
	It is known that if no two edges share more than one common vertex,
	and $d$ is sufficiently large as a function of $r$, then there is a nearly
	perfect matching (see, e.g.,~\cite{AKS}). On the other hand, if a typical pair of
	intersecting edges
	has more than one common vertex  then the typical degree of an
	element of the line graph is significantly less than $1+(d-1)r$,
	hence one could expect that the above estimate is not too close to
	being tight for large degrees. Surprisingly we show that this is
	nearly tight, even if we assume that the hypergraph is $r$-uniform,
	$r$-partite, has no multiple edges and is regular of arbitrarily
	high degree.
	\begin{claim}
		\label{cl:p11}
		For every prime power $p$ and every positive integer  $m$ there is
		an $r=p+2$ uniform, $r$-partite hypergraph, which is regular of
		degree $d=p^2 m$, in which the number of vertices in each of the $r$
		vertex classes is $pm$ (hence the total number of vertices is
		$n=p(p+2)m$), there are no multiple edges,
		and  the maximum size of a matching is
		$m$, namely, it covers only a $1/p=1/(r-2)$ fraction of the
		vertices.
	\end{claim}
	Note that by the known results about the distribution of primes
	this implies that for every large $r$ there is an $r$-uniform
	$r$-partite $d$-regular hypergraph, for arbitrarily large $d$, in
	which no matching covers more than a $(1+o(1))\frac{1}{r}$ fraction
	of the vertices.
	Indeed, for every large $r$ there exists an $r'$ such that $r - o(r) \leq r' \leq r$ and $r'-2$ is a prime power. We can then take the construction for $r'$ guaranteed by Claim~\ref{cl:p11}, replace one of the $r'$ parts with $r-r'+1$ copies of it, and for each edge replace the vertex from that part with all new vertices corresponding to it. We get an $r$-uniform $r$-partite hypergraph of the same regularity and with the same set of matchings. By the choice of $r'$, we get that still only a $(1+o(1))\frac{1}{r}$ fraction of the vertices are covered by the largest matching. This bound is tight up to the $o(1)$ term, by the	argument initiating this discussion.
	\vspace{0.2cm}
	
	\noindent
\begin{proof}
    A projective plane of order $p$ minus a point corresponds to a $(p+1)$-uniform,
	$(p+1)$-partite hypergraph, with $p$ vertices in each vertex class,
	which is $p$ regular, has $p^2$ edges, and every two edges
	intersect. Indeed, the vertex classes are just the sets
	$L-x$ where $x$ is the deleted point and $L$ is any line
	containing it. The edges are all other lines.
	
	Take $m$ vertex disjoint copies $P_1, \ldots ,P_m$ of the above
	hypergraph to get a $(p+1)$-uniform $(p+1)$-partite hypergraph with
	vertex classes $V_1, \ldots ,V_{p+1}$, each of size $pm$, in which
	the maximum matching contains $m$ edges. Add another vertex class
	$V_{p+2}$ of size $pm$. Now let $H$ be the $(p+2)$-uniform
	$(p+2)$-partite hypergraph with vertex classes
	$V_1, \ldots V_{p+2}$
	whose edges are all sets $L \cup v$ for all the pairs
	$(L,v)$ where $L$ is an edge of some $P_i$ and $v \in V_{p+2}$.
	It is easy to check that this hypergraph satisfies all properties
	in the proposition.
\end{proof}		
	\bigskip
	
	\section{Concluding remarks and open problems}\label{sec:concluding}
	Question~\ref{qst:AllNice}, that is, the possible conjecture  that
	for any finite family of graphs $\cF$ that contains at least one
	forest, there exists a constant $k_0=k_0(\cF)$ such that $\cF$ is
	$k$-nice for all $k \geq k_0$, remains wide open.
	
	In the previous sections we showed that the assertion of this
	question holds for any family consisting of two connected graphs,
	each having three edges. As mentioned in the introduction, it also
	holds for any family consisting of one star as well as for $\cF =
	\{P_4$\}~\cite{GoodGraph} and for any family consisting of a single
	matching~\cite{BG}.
	
	Here are several comments about additional families for which this
	assertion holds. If $K_2 \in \cF$ then, trivially, using the
	notation of Section \ref{sec:nice}, $c_k(\cF)=g_k(\cF)=1$ for every
	$k$.
	
	If one of the members of $\cF$ is a path with two edges and no other
	member of $\cF$ is a matching, then it is easy to see that $c_k(\cF)
	= g_k(\cF) = k + (k \mod 2)$ holds for all $k$. Indeed, let $G$ be a
	graph satisfying $\chi(G) = g_k(\cF)$ and let $c$ be a $k$-coloring
	of $E(G)$ with no monochromatic copy of any member of $\cF$.
	Clearly, any color class of $c$ is a matching, implying that the
	maximum degree of $G$ is at most $k$. Hence, by Brooks' theorem
	(Theorem~\ref{thm:brooks}), the chromatic number of $G$ is $k+1$ if
	and only if $G \cong K_{k+1}$ (or an odd cycle for $k=2$), and the
	desired result follows from the known values of the chromatic index
	of complete graphs.
	
	If one of the members of $\cF$ is a path with two edges and another is
	a matching with $r+1$ edges, then any color class is a
	matching of size at most $r$. It is easy to see that for large $k$
	the maximum size of a complete graph that can be covered by
	$k$ such matchings is exactly $\max \{s: rk \geq {s \choose 2}\}$.
	Indeed, one direction is trivial. For the other, if $rk \geq {s \choose
		2}$
	and $k$ is large as a function of $r$, then the edges of $K_s$ can
	be decomposed into $k$ matchings (possibly some are empty and some are
	large), and it is well-known that if a graph can be decomposed into
	$k$ matchings it can also be decomposed into $k$ matchings of nearly
	equal
	sizes (see, e.g.,~\cite{BM}). Since any graph with chromatic number $s$
	has at least
	${s \choose 2}$ edges, $g_k(\cF) < s+1$, and thus $\cF$ is $k$-nice for
	all large $k$.
	
	If one of the members of $\cF$ is a matching of size $2$ and no
	other member is a star, then it is easy to verify that
	$c_k(\cF)=g_k(\cF)=k+1$ for all $k$, as here each color class must
	be a star. If $\cF$ contains a matching of size $2$ and a
	star with $r+1$ edges, then $c_k(\cF)=g_k(\cF)=\max\{s: rk \geq {s
		\choose 2}\}$
	for all large $k$. To show this, we need the fact that, if $rk \geq {s
		\choose 2}$
	and $k$ is large as a function of $r$, then
	the complete graph $K_s$ can be covered by $k$ stars, each of size at
	most
	$r$. This follows, for example, from a very special case of
	Gustavsson's Theorem~\cite{Gu}. We omit the details.
	
	The Erd\H{o}s-S\'os conjecture, raised in 1962 (see, for
	example,~\cite{W}), asserts that every graph with average degree
	exceeding $r-1$, must contain as a subgraph every tree with $r$
	edges. This is known to be true in many cases. If the conjecture
	holds for some tree $T$ with $r$ edges, then for $\cF=\{T\}$, the
	average degree of a graph whose edges can be colored by $k$
	colors  with no monochromatic copy of $T$ is at most $k(r-1)$. Thus
	each such graph has chromatic number at most $k(r-1)+1$ (as follows
	by considering its $k(r-1)$-core), and equality can hold if and only
	if the complete graph on $k(r-1)+1$ vertices can be colored as
	above. By results of Ray-Chaudhuri and Wilson~\cite{RCW2} about
	resolvable designs, if $s \md{r}{r(r-1)}$ and $s$ is sufficiently
	large, then $K_s$ can be decomposed into subgraphs, each being the
	vertex disjoint union of $s/r$ cliques of size $r$. As each such
	subgraph cannot contain a tree on $r+1$ vertices, this shows that
	for all large $k \md{1}{r}$, $c_k(\cF)=g_k(\cF)=k(r-1)+1$ for
	$\cF=\{T\}$, provided we know that the Erd\H{o}s-S\'os conjecture
	holds for $T$. The same applies, of course, to any family $\cF$
	containing such a tree $T$ as well as any additional members as long
	as each of them contains a connected component of size at least
	$r+1$. Note that this gives many additional families for which the
	assertion of Conjecture~\ref{conj:ManyNice} holds.
	
	It is worth noting that the condition that $k \geq k_0(\cF)$ in
	Question~\ref{qst:AllNice} is necessary. Indeed, consider, for
	example, a family $\cF$ that contains two members, a star with $t+1$
	edges and $K_3$, where $t$ is large. If $t > R_k(K_3)$, then
	$c_k(\cF)=R_k(K_3)-1$ as this is the maximum number of vertices in a
	complete  graph whose edges can be colored by $k$ colors with no
	monochromatic triangle. On the other hand, $g_k(\cF)$ here is at
	least the maximum chromatic number of a triangle-free graph with
	maximum degree at most $kt-1$, as each such graph can be decomposed,
	using Vizing's Theorem, into $k$ triangle-free subgraphs, each with
	maximum degree at most $t$. It is well known that this maximum
	chromatic number is $\Theta(kt/\log (kt))$, which can
	be much larger than $R_k(K_3)$ for large $t$ (see, e.g.,~\cite{AS,Johansson}).
	
	The question considered here may be generalized to oriented graphs
	(that is, directed graphs with no cycles of length $2$). For a
	family of oriented graphs $\cF$ let $p_k(\cF)$ denote the maximum
	number of vertices of a tournament whose edges can be colored by $k$
	colors with no monochromatic copy of any member of $\cF$. Let
	$q_k(\cF)$ denote the maximum possible chromatic number of an
	oriented graph whose edges can be colored with $k$ colors with no
	monochromatic copy of any member of $\cF$. Call $\cF$ $k$-nice if
	$p_k(\cF)=q_k(\cF)$. The existence of graphs of high girth and high
	chromatic number implies that a family that contains no oriented
	forest is not $k$-nice for any $k$. It is not difficult to show that
	the family consisting of any single directed path is $k$-nice for
	every $k$. This can be done using the theorem which asserts that any oriented
	graph of chromatic number $g$ contains a directed path with $g$
	vertices, a statement proved independently by Gallai~\cite{Ga},
	Roy~\cite{Roy}, Hasse~\cite{Ha} and Vitaver~\cite{Vi}. One may study
	the directed analogues of Question~\ref{qst:AllNice} and
	Conjecture~\ref{conj:ManyNice}, which are the following.
	\begin{question}\label{qst:OrientedNice}
		Is it true that for any finite family of oriented graphs $\cF$ that
		contains at least one forest, there exists a constant $k_0=k_0(\cF)$
		such that $\cF$ is $k$-nice for all $k \geq k_0$ ?
	\end{question}
	
	\begin{question}\label{qst:OrientedManyNice}
		Is it true that any finite family of oriented graphs $\cF$ that
		contains at least one forest is $k$-nice for infinitely many
		integers $k$?
	\end{question}
	See \cite{Yu} for (somewhat) related results.\\
	
	Finally,
	the topics discussed
	in this paper
	suggest
	a generalization of Ramsey numbers: For (finite or
	infinite) families of graphs $\cF_1,\dots,\cF_k$ let
	$R=R(\cF_1,\dots,\cF_k)$ denote the smallest integer $R$ so that for
	any $k$-coloring of the edges of the complete graph $K_R$ there is a
	monochromatic copy of some member $F \in \cF_i$ colored in color
	number $i$ for some $1 \leq i \leq k$.
	
	The study of the number $\g_k^{}(\cF)$ presented in Section~\ref{sec:nice} is also interesting. As far as we know, no work has been previously done to determine the value of $g_k(\cF)$.  It is easy to see that  $\g_k^{}(\cF)$ is finite if and only if the family $\cF$ contains a forest. Indeed, if the chromatic number of a graph is at least $r$, then it contains a subgraph with minimum degree at least $r-1$. Hence in any $k$-edge-coloring there is a monochromatic subgraph with average degree at least $(r-1)/k$, and hence also a subgraph with minimum degree at least half of that. That is, for every graph $G$ with $\chi(G)=r$ and for every $k$-coloring of $E(G)$, there exists a monochromatic subgraph with minimum degree at least $(r-1)/2k$. This monochromatic graph contains every forest on at most $(r-1)/2k$ vertices. Thus, if $n_0$ is the order of the smallest forest in $\cF$, then $\g_k^{}(\cF)\leq 2kn_0$. The other direction follows from the well-known argument by Erd\H{o}s \cite{Erdos1959}  that for any
	two positive integers $\chi,g$ there exists a graph $G$ with
	$\chi(G)=\chi$ and girth greater than $g$.
	
	For directed graphs there are early papers dealing with
	the analogous problem for $g_k(\cF)$. For instance, Chv\'atal proved   \cite{MonoPath}  that for a directed path $P$ of length $\ell$ we have $\g_k^{}(P)=\ell^k$. See also   \cite{MonoWalk} for related results.
	

	\vspace{0.2cm}
	
	\noindent\textbf{Acknowledgement}
	The research on this project was initiated during a joint
	research workshop of Tel Aviv University and the Free University
	of Berlin on
	Graph and Hypergraph Coloring Problems,
	held in Tel Aviv in March 2017, and supported by a GIF grant number G-1347-304.6/2016. We
	would like to thank both institutions for their support. We would also like to thank the anonymous referees for helpful comments.

	{}
	
\end{document}